\documentclass[11pt,leqno,letterpaper]{article}

\usepackage{amsmath}
\usepackage{amsthm}
\usepackage{amssymb}
\usepackage{graphicx}
\usepackage[hidelinks]{hyperref}

\usepackage[hmargin=2.5cm,vmargin=2.5cm]{geometry}

\usepackage{homotopy_l1_arxiv}  

\usepackage{breakcites}

\title{The Homotopy Method Revisited: Computing Solution Paths of \( \ell_1 \)-Regularized Problems}
\author{Björn Bringmann, Daniel Cremers, Felix Krahmer, Michael Möller}

\DeclareMathOperator{\LS}{LS}

\begin{document}

\maketitle

\begin{abstract}
\( \ell_1 \)-regularized linear inverse problems are frequently used in signal processing, image analysis, and statistics. The correct choice of the  regularization parameter \( t \in \mathbb{R}_{\geq 0} \) is a delicate issue. Instead of solving the variational problem for a fixed parameter, the idea of the homotopy method is to compute a complete solution path \( u(t) \) as a function of \( t \). In a celebrated paper by Osborne, Presnell, and Turlach, it has been shown that the computational cost of this approach is often comparable to the cost of solving the corresponding least squares problem. Their analysis relies on the one-at-a-time condition, which requires that different indices enter or leave the support of the solution at distinct regularization parameters. In this paper, we introduce a generalized homotopy algorithm based on a nonnegative least squares problem, which does not require such a condition, and prove its termination after finitely many steps. At every point of the path, we give a full characterization of all possible directions. To illustrate our results, we discuss examples in which the standard homotopy method either fails or becomes infeasible. 
To the best of our knowledge, our algorithm is the first to provably compute a full solution path for an arbitrary combination of an input matrix and  a data vector.
\end{abstract}



\section{Introduction}

In recent years, sparsity promoting regularizations for inverse problems played an important role in many fields such as image and signal analysis \cite{Rudin1992,Candes2006a} and statistics \cite{Tibshirani1996}.
The typical setup is that one tries to recover an unknown signal \( \hat{u} \in \mathbb{R}^N \) from linear measurements \( { f = A\hat{u} + \eta \in \mathbb{R}^m } \) under the model assumption that \( \hat{u} \) is approximately sparse, i.e., has only few significant coefficients. Here \( \eta \) is  typically small and represents measurement noise. A common approach is to minimize the energy functional
\begin{equation*}
u \mapsto  E_t(u):= \half \| Au -f \|_2^2 + t \| u \|_1 ~,
\end{equation*}
where \( A\in \mathbb{R}^{m\times N} \) is a matrix,\(~f \in \mathbb{R}^m \) the data vector, and \( t > 0 \) is a regularization parameter. In the statistics literature, this method is called the Lasso \cite{Tibshirani1996},
while the inverse problems literature mostly refers to it as \( \ell_1 \)-regularization. The choice of the regularization parameter \( t  \) often proves to be difficult. While choosing  a small \( t \) yields unnecessarily noisy reconstructions,  choosing a large \( t \) diminishes the features of the original signal \( \hat{u} \). An approach to this problem is to compute a minimizer \( u(t) \) for every \( t > 0 \), and subsequently choose a suitable regularization parameter, for instance by visual inspection of this family of solutions. \\
  
A popular algorithm to compute the full solution path is the so-called homotopy method. 
The homotopy method is based on the observation that there always exists a piecewise linear and continuous solution path \( t \mapsto u(t) \), see Figure \ref{hom_intro:fig_example} for an example. The main idea of the homotopy method is to start at a large parameter \( t^0 \), so that the unique solution is \( u(t^0)= 0 \), and then follow the solution path in the direction of decreasing \( t \). At every kink of the solution path, the classical homotopy method \cite{Osborne2000,Efron2004,Donoho2008} computes a new direction by solving a linear system. As it turns out, the computational cost of the classical homotopy method is often comparable to the cost of solving a single minimization problem \( u(t) \in \argmin_u E_t(u) \) or solving the least squares problem \( u = A^\dagger f \). 
For this reason the homotopy method has proven to efficiently compute reconstructions when the noise level is unknown - provided the output is really a solution path.
This has been shown to be the case given a so-called \emph{one-at-a-time} condition \cite{Osborne2000,Efron2004,Tibshirani2013}, cf. Definition \ref{alg_hom_other:def_one_at_a_time}. Loosely speaking, this conditions requires that at every kink only one index joins or leaves the support of \( u(t) \). The first works \cite{Osborne2000,Efron2004} additionally required the uniqueness of the solution path \( t\mapsto u(t) \), i.e., 
they required that \( u(t) \) is the only minimizer of \( E_t \) for every \( t > 0 \). In \cite{Tibshirani2013}, the homotopy method is extended to the case of non-uniqueness; again the analysis implicitly assumes the one-at-a-time condition (see Section  \ref{subsection:standard_homotopy} for a detailed discussion). \\ 
\begin{figure}
\begin{centering}
\includegraphics[width=0.5\textwidth,keepaspectratio=true]{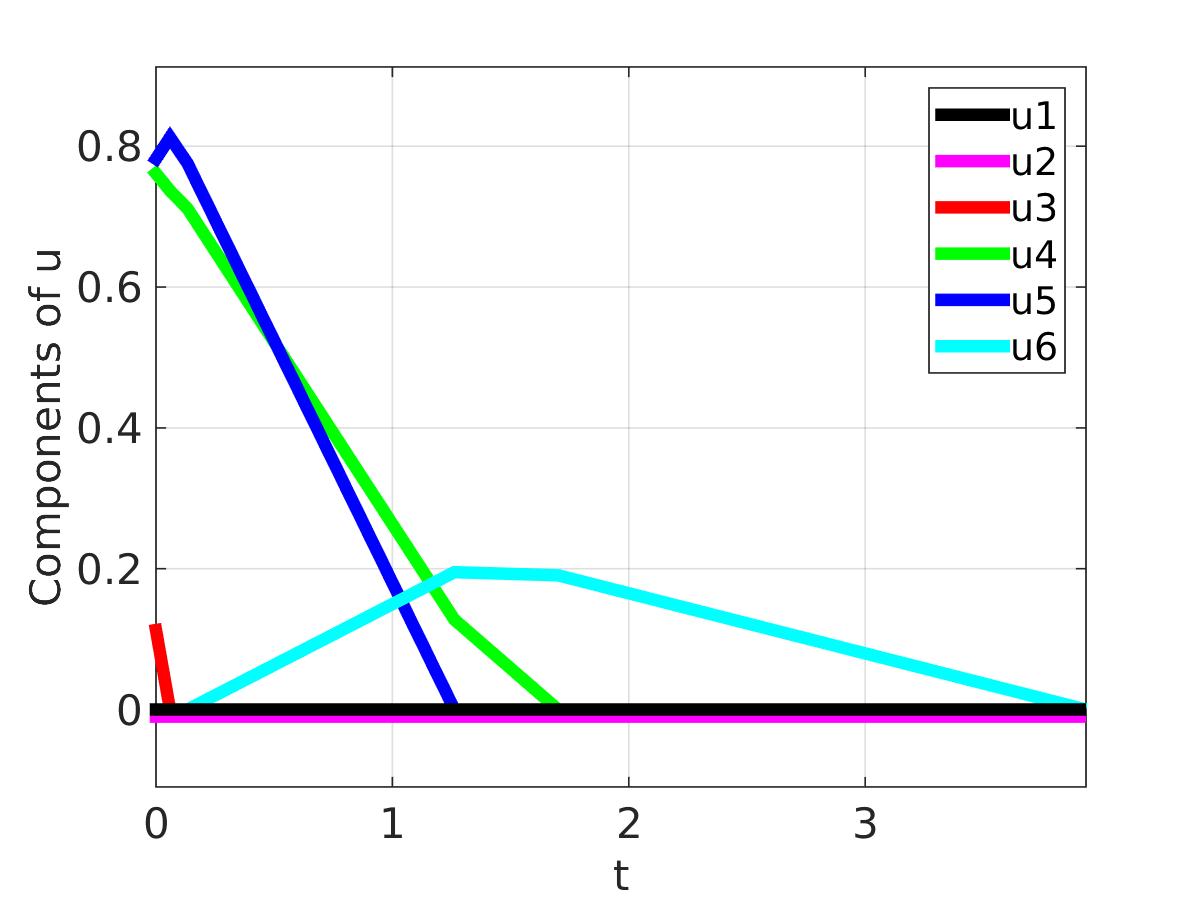}
\caption{We plot the different components  of \( u(t) \) corresponding to \( A \in \mathbb{R}^{3\times 6} \) and \(  f \in \mathbb{R}^3 \) with i.i.d. standard Gaussian entries  \( A_{ij} \) and \(f_i \). The solution path \( u(t) \) is piecewise linear.}
\label{hom_intro:fig_example}
\end{centering}
\end{figure}

The one-at-a-time condition is known to hold in various scenarios. 
For example, empirical observations indicate that it is true with high probability for input \({ A\in \mathbb{R}^{m\times N} }\) and \( f \in \mathbb{R}^m \)
drawn from independent continuous probability distributions. Also uniqueness of the solution path holds in many cases. 
Necessary and sufficient conditions for the uniqueness of minimizers have been established in  \cite{Foucart2013,Zhang2013,Zhang2014}. In \cite{Tibshirani2013}, it has been shown that, if the columns of \( A \) are independent and drawn from continuous probability distributions, the minimizer \( u(t) \) is almost surely unique for every \( f \in \mathbb{R}^m \) and \( t \in \mathbb{R}_{> 0} \).\\
However, both the one-at-a-time condition and uniqueness are known to be violated in certain cases \cite{Loris2008,Tibshirani2013}.  For instance, when the entries  of \( A \) and \( \hat{u} \)  are chosen as independent random signs and the measurements are exact, i.e., \( \eta= 0 \), the one-at-a-time condition is regularly violated. In such cases it has been observed that standard homotopy implementations can fail to find a solution path \cite{Loris2008}. If in addition uniqueness is violated, even the finite termination property of the homotopy method may no longer 
hold, see Proposition \ref{hom_alg:prop_infinite_kinks}. \\

In this paper, we propose a generalized homotopy method, which addresses these issues. In contrast to the classical homotopy method \cite{Osborne2000,Efron2004}, which solves a linear system at each kink, the generalized homotopy method solves a nonnegative 
least squares problem. The main result of this paper, Theorem \ref{alg_hom:thm_main}, shows that this new algorithm always computes a full solution path in finitely many steps, even without a  one-at-a-time assumption.
Along the way, we give a full characterization of all directions which linearly extend a given partial solution path, see Theorem \ref{directions:thm_characterization}. Our characterization is of interest even under the one-at-a-time condition, since it provides a unified treatment of both hitting and leaving indices (cf.~\cite{Efron2004}).
We also show that, under the assumptions of \cite{Osborne2000,Efron2004}, the generalized homotopy method and the standard homotopy method coincide.

\subsection{Outline}
In Section \ref{section:notation_hom} we set up our notation and recall some  basic facts commonly used in the sparse recovery literature. The set of all possible directions (cf. Definition \ref{directions:def_D}) is characterized in Section  \ref{section:hom_possible_directions}. In Section \ref{section:hom_alg} we 
propose the generalized homotopy method, and prove that it always computes a solution path.  In Section \ref{section:hom_alg_other}
we compare the generalized homotopy method with the standard homotopy method and the adaptive inverse scale space method \cite{Burger2012}.

\section{Notation and Background}\label{section:notation_hom}

For  \( A \in \mathbb{R}^{m \times N} \) we will denote the \( i^{\operatorname{th}} \) column of \( A \) by \( A_i \in \mathbb{R}^N \). Similary, for a subset \( \mathcal{S} \subseteq [N] \),  \( A_\mathcal{S} \in \mathbb{R}^{m \times |\mathcal{S}| }\) is the submatrix of \( A \) with columns indexed by \( \mathcal{S} \). Furthermore, with a slight abuse of notations, we write \( A_\mathcal{S}^T = (A_\mathcal{S})^T \). The pseudoinverse of \( A \) is denoted by \( A^\dagger \). \\
For \( t \in \mathbb{R}_{\geq 0 } \) and \( u \in \mathbb{R}^N \), the equicorrelation set \( \mathcal{E}(t,u) \) is defined as
\begin{equation}\label{notation:eq_def_equicorrelation}
\mathcal{E}(t,u):= \{ i \in [N]\colon |A_i^T ( Au -f ) |= t \}~.
\end{equation}
Indeed, for least squares solutions  \( u_{\LS} \in  \argmin_u \| A u -f \|_2^2 \), we have that \( A^T A u_{\LS} = A^T f \). 
Even though this equation is no longer true for solutions of the \( \ell_1 \)-regularized problem, it turns out to be useful to distinguish indices \( i \in [N] \) according to the magnitude of \( A_i^T (Au -f ) \).
The active set \( \mathcal{A}(u) \) is the support of \( u \), i.e.,  
\begin{equation*}
\mathcal{A}(u):=\operatorname{supp}(u)= \{ i \in [N]\colon u_i \not  = 0 \}~.
\end{equation*}
For a fixed regularization parameter \( t \geq 0 \) and vector \( f \in  \mathbb{R}^m \), we define the set of minimizers \( U_t(f) \) by
\begin{equation}\label{notation:eq_def_U_t}
U_t(f) :=
\begin{cases}
\begin{tabular}{cl}
 \( \argmin_{u \in \mathbb{R}^N } \half \| A u - f \|_2^2 + t ~ \| u \|_1 \) & if \( t>0 \) \\
 \( \argmin_{u \in \mathbb{R}^N} \| u \|_1 \)  \quad s.t. \( A^T(Au-f) = 0 \) &if \( t=0 \)
 \end{tabular}
 \end{cases}~.
\end{equation}
We will often drop the dependence on \( f \) and simply write \( U_t \).\\

We recall some basic facts about the variational problem \eqref{notation:eq_def_U_t}. A proof is included for the reader's convenience.
\begin{lemma}[\cite{Zhang2013}]\label{notation:lem_basic_facts}
Let \( u_1(t),u_2(t) \in U_t \) be two minimizers. Then one has that:
\begin{enumerate}[label=(\alph*),leftmargin=*] 		
\item \(  Au_1(t) = Au_2(t)\); \label{notation:enum_image_equality}
\item  \label{notation:enum_subgradient}
for \( t > 0 \), the map \( p \) given by 
\end{enumerate}
\begin{equation}\label{notation:eq_the_subgradient}
p(t):=\frac{1}{t} A^T(f-Au(t)) 
\end{equation}
\begin{description}[leftmargin=4.4ex,labelindent=3.2ex]
\item[] satisfies \( p(t) \in \partial \| u(t) \|_1  \) for all \( t > 0 \) and is independent of the specific choice of \( u(t) \in U_t \);
\end{description}
\begin{enumerate}[label=(\alph*),leftmargin=*,resume]
\item \label{notation:enum_act_equi}
\( \| A^T ( f - Au(t) ) \|_\infty \leq t \quad \text{and} \quad \mathcal{A}(u(t))\subseteq \mathcal{E}(t):= \mathcal{E}(t,u(t)) ~. \)
\end{enumerate}
\end{lemma}

\begin{remark} 
In the following,  \( p(t) \) always  refers to the subgradient  given by \eqref{notation:eq_the_subgradient}.
\end{remark}

\begin{proof}
For \( t= 0 \), \ref{notation:enum_image_equality} holds since the constraint \( A^T A u_1(t) = A^T A u_2(t) \) implies \( u_1(t) - u_2(t) \in \Ker(A^T A ) = \Ker(A) ~.\)  \\ 
For \( t > 0 \), \ref{notation:enum_image_equality} follows from the strict convexity of \( \|\cdot \|_2^2 \). Indeed, set \( v := \half u_1(t) + \half u_2(t) \). Then 
\begin{align}
&~\half \| A v - f \|_2^2 + t \| v \|_1  \nonumber \\
&\leq \half \left( \half \| Au_1(t) -f \|_2^2 + t \| u_1(t) \|_1 \right) +  \half \left( \half \| Au_2(t) -f \|_2^2 + t \| u_2(t) \|_1 \right)\label{directions:eq_basic_facts_proof_1}   \\~
&= \min_{u \in \mathbb{R}^N} \left( \half \| Au -f \|_2^2 + t \| u \|_1 \right) \nonumber~.
\end{align}
If  \( Au_1(t) \not =Au_2(t) \) , the inequality \eqref{directions:eq_basic_facts_proof_1} would be strict, leading to a contradiction. As a result,~\( Au(t) \) is independent of the minimizer chosen and \( \mathcal{E}(t) = \mathcal{E}(t,u(t)) \) is well-defined. \par
The statements \ref{notation:enum_subgradient} and \ref{notation:enum_act_equi} are an immediate consequence of the optimality condition 
\begin{equation*}
0 \in A^T(Au(t) -f ) + t ~\partial \| u \|_1 ~,
\end{equation*}
since the subdifferential of the \( \ell_1 \)-norm is given by 
\begin{equation*}
\partial \| u \|_1 = \{ p \in \mathbb{R}^N \colon p_i = \sgn u_i ~ \forall i \in \mathcal{A}(u)~ \text{and} ~ |p_i|\leq 1 ~ \forall i \not \in \mathcal{A}(u) \} ~.
\end{equation*} 
\end{proof}

\section{The Set of Possible Directions \( \Directions \)}\label{section:hom_possible_directions}

We aim to construct a piecewise linear and continuous function \( u\colon \mathbb{R}_{\geq 0 } \rightarrow \mathbb{R}^N \) satisfying 
\begin{equation}\label{directions:eq_variational_problem}
u(t) \in U_t = \argmin_{u \in \mathbb{R}^N } \half \| A u - f \|_2^2 + t ~ \| u \|_1 \quad \forall t > 0 
\end{equation}
and \( u(0) \in U_0 \). To this end we make the following ansatz. \\
Assume we already have a solution \( u(\fixedt) \in U_{\fixedt} \). Set \( u(t) = u(\fixedt) + ( \fixedt - t )~d \) and try to choose \( d \in \mathbb{R}^N \) such that  \( u(t) \in U_t \) for all \(  t \in [\fixedt-\delta,\fixedt] \), where \( \delta=\delta(\fixedt,u(\fixedt),d) > 0 \). This motivates the definition of the set of all possible directions \( \Directionsft \).

\begin{definition}\label{directions:def_D}
 Let \( \fixedt> 0 \) be a regularization parameter and \( u(\fixedt) \in U_{\fixedt} \) be a solution of the variational problem \eqref{directions:eq_variational_problem}. 
The set of all possible directions \( \Directionsft \) is defined as 
 \begin{equation*}
 \Directionsft = \left\{ d \in \mathbb{R}^N \colon \exists \delta \in (0,\hat{t}]~\text{s.t.} ~ u(t)=u(\fixedt)+(\fixedt-t) d \in U_t ~ \forall t \in [\fixedt-\delta,\fixedt] \right\}~.
 \end{equation*}
\end{definition}
We now state and prove the main theorem of this section.
\begin{theorem}\label{directions:thm_characterization}
The set of possible directions  \( \Directionsft \) at \( \left(\fixedt,u(\fixedt) \right) \) is the set of solutions to a nonnegative least squares problem. More precisely, set \begin{equation*}
 r(\fixedt)= f - Au(\fixedt) \quad \text{and} \quad p(\fixedt) = \frac{1}{\fixedt} A^T \left( f - Au(\fixedt) \right) ~.
 \end{equation*} 
Then we have that 
 \begin{equation}\label{directions:eq_characterization}
 \begin{aligned}
 \Directionsft = \argmin_{d \in \mathbb{R}^N} \| Ad - \frac{1}{\fixedt} r(\fixedt) \|_2^2 \quad \text{s.t.} \quad  d_i ~ p(\fixedt)_i &\geq 0 ~\forall i \in \mathcal{E}(\fixedt)\backslash \mathcal{A}(u(\fixedt))~, \\
 d_i &= 0 ~ \forall i \not \in \mathcal{E}(\fixedt) ~.
 \end{aligned}
\end{equation}
\end{theorem}
\begin{remark}\label{directions:rem_sign_constraint}
The major difference to the standard homotopy method (cf. Algorithm \ref{alg_hom_other:standard_algorithm}) is the condition \( d_i p(\fixedt)_i \geq 0 \) for all \( i \in \mathcal{E}(\fixedt) \backslash \mathcal{A}(u(\fixedt)) \).
In fact, if \( u(t) \in U_t \) for all \( t \in [\fixedt-\delta,\fixedt] \), then the direction \( d \) necessarily satisfies this condition.
To see this, let \( i \in \mathcal{E}(\fixedt) \backslash \mathcal{A}(u(\fixedt)) \) and \( p(\fixedt)_i = 1 \). It follows that \( u(t)_i \geq 0 \) for all \( t \in [\fixedt-\delta,\fixedt] \) because \( p(t) \) is continuous and \( p(\fixedt)_i = 1 \). Therefore \( u(\fixedt)_i + ( \fixedt - t ) d_i \geq 0 \), which, since \( i \not \in \mathcal{A}(u(\fixedt)) \), implies that \( p(\fixedt)_i d_i = d_i \geq 0 \). \\
As we will discuss in Section \ref{subsection:standard_homotopy} below, this condition is sometimes violated for directions computed by the standard homotopy method; so an extra condition is indeed necessary.
\end{remark}
\begin{remark}
By a change of variable, the constraints \( p(\fixedt)_i d_i \geq 0 \) are easily transformed into nonnegativity constraints, which makes \eqref{directions:eq_characterization} a nonnegative least squares problem. 
\end{remark}
\begin{proof}
The strategy of the proof is as follows: We characterize the solutions of the nonnegative least squares problem by the Karush Kuhn Tucker (KKT) conditions (equations \eqref{directions:eq_KKT_mult}-\eqref{directions:eq_KKT_ineq_4}), and compare them componentwise to a characterization of the set of possible directions. \\
Throughout the proof, let \( \mathcal{E}:=\mathcal{E}(\fixedt),~ \mathcal{A}:=\mathcal{A}(u(\fixedt)),~\text{and}~ \Directions:=\Directionsft \). \\

Let us start by stating the KKT conditions for the nonnegative least squares problem in  \eqref{directions:eq_characterization}:
\(d\in \mathbb{R}^N \) is a minimizer of \eqref{directions:eq_characterization} if and only if
there exist \( \lambda,\theta \in \mathbb{R}^N \) such that 
 
\begin{align}
A^T A d - p(\fixedt) + \lambda + \theta &= 0 ~, \label{directions:eq_KKT_mult} \allowdisplaybreaks \\
d_{\mathcal{E}^C} &= 0		~, \label{directions:eq_KKT_eq_1}\\
\theta_{\mathcal{E}} &= 0	~, \label{directions:eq_KKT_eq_2} \allowdisplaybreaks \\
d_i ~ p(\fixedt)_i & \geq 0 \quad \forall i \in \eac ~,  \label{directions:eq_KKT_ineq_1}\\
\lambda_{ ( \mathcal{E}\backslash\mathcal{A} )^C } &= 0  ~, \label{directions:eq_KKT_ineq_2}\\
\lambda_i ~ p(\fixedt)_i &\leq 0  \quad \forall i \in \eac~,		\label{directions:eq_KKT_ineq_3}		\\
\lambda_i ~ d_i &= 0 \quad \forall i \in \eac \label{directions:eq_KKT_ineq_4}~.
\end{align}
We now show that every solution \( d \) of this system is a possible direction. We need to prove that there exists a \( \delta > 0 \) such that \( u(t)= u(\fixedt)+ ( \fixedt - t ) d \in U_t \)  for all \( t \in [\fixedt-\delta,\fixedt] \). Recalling the optimality condition 
\begin{equation*}
0 \in A^T ( Au(t) - f ) + t ~ \partial \| u(t) \|_1 ~, 
\end{equation*}
it suffices to show that \( p(t) := \frac{1}{t} A^T ( f - A u(t) ) \in \partial \| u(t) \|_1  \). \\
We begin by rewriting \( p(t) \). By inserting the definition of \( u(t) \), it follows that
\begin{equation}\label{directions:eq_p_rewrite}
\begin{aligned}
p(t)&= \frac{1}{t} A^T ( f - Au(\fixedt) - (\fixedt -t ) Ad ) \\
	&= \frac{\fixedt}{t} p(\fixedt) - \frac{\fixedt-t}{t} A^T A d \\
	&= p(\fixedt) + \frac{\fixedt-t}{t} ( p(\fixedt) - A^T A d ) ~.
\end{aligned}
\end{equation}
We argue componentwise proving that \( p(t)_i \in \partial | u(t)_i | \) for all \( i \in [N] \). We distinguish the 
three different cases \( i \in \mathcal{A} \), \( i \in \mathcal{E} \backslash \mathcal{A} \), and \( i \in \mathcal{E}^C \). \par \vspace{1ex}
\emph{Case 1: \( i \in \mathcal{A} \).} Since \( u(t) \) is continuous, the equality \( \sgn(u(t)_i)=\sgn(u(\fixedt)_i) \) holds for all \( t \) in some small interval  \( [\fixedt-\delta_\mathcal{A},\fixedt] \).
Using \eqref{directions:eq_KKT_mult},\eqref{directions:eq_KKT_eq_2}, and \eqref{directions:eq_KKT_ineq_2}, it follows that 
\begin{equation*}
p(t)_i = p(\fixedt)_i + \frac{\fixedt-t}{t} \left( p( \fixedt ) - A^TAd \right)_i = p(\fixedt)_i ~.
\end{equation*}\par
\emph{Case 2: \( i \in \eac \).} It follows that 
\begin{equation*}
u(t)_i = u(\fixedt)_i + ( \fixedt- t ) d_i \overset{\eqref{directions:eq_KKT_ineq_1}}{=} ( \fixedt -t ) |d_i | p(\fixedt)_i ~.
\end{equation*}
From \eqref{directions:eq_KKT_mult}, \eqref{directions:eq_KKT_eq_2}, and \eqref{directions:eq_p_rewrite}, we deduce that 
\begin{equation*}
p(t)_i = p(\fixedt)_i + \frac{\fixedt-t}{t} \lambda_i ~.
\end{equation*}
If \( d_i \not = 0 \), then by complementary slackness \eqref{directions:eq_KKT_ineq_4}, \( ~ \lambda_i = 0 \) holds. Therefore  \( p(t)_i = p(\fixedt)_i \) and \( \sgn(u(t)_i) = p(\fixedt)_i = p(t)_i \). \\
If \( d_i = 0 \), we have that  \( u(t)_i = 0 \).  Since \( \lambda_i ~ p(\fixedt)_i \leq 0 \), it follows  that for some  \( \delta_{\mathcal{E}\backslash \mathcal{A}} > 0 \)
\begin{equation*}
|p(t)_i| = | p(\fixedt) + \frac{\fixedt-t}{t} \lambda_i |\leq 1 \quad \forall t \in [\fixedt- \delta_{\mathcal{E}\backslash \mathcal{A}} ,\fixedt]~.
\end{equation*} \par 
\emph{Case 3: \( i \in \mathcal{E}^C \).} Lemma \ref{notation:lem_basic_facts} and equation \eqref{directions:eq_KKT_eq_1}  yield \( u(\fixedt)_{\mathcal{E}^C} = 0 \) and \( d_{\mathcal{E}^C} = 0 \). Thus, it follows that \( u(t)_i = 0 \). Since \( |p(\fixedt)_i |< 1 \) and \( p(t) \) is continuous, there exists a \( \delta_{\mathcal{E}^C} > 0 \) such that we have \( |p(t)_i|\leq 1 \) for all \( t\in [\fixedt-\delta_{\mathcal{E}^C},\fixedt] \). \\

Setting \( \delta = \min \{ \delta_\mathcal{A}, \delta_{\mathcal{E} \backslash \mathcal{A} }, \delta_{\mathcal{E}^C} \} \), we conclude that \( p(t) \in \partial \| u(t) \|_1 \) for all \( t \in [\fixedt-\delta,\fixedt] \), and hence that
\( d \) is a valid direction.  ~ \vspace{2ex} \\

It remains to show that every possible direction \( d \) is a solution to the nonnegative least squares problem. To this end, we show that \( d \) satisfies the KKT conditions \eqref{directions:eq_KKT_mult}-\eqref{directions:eq_KKT_ineq_4}. Set 
\begin{alignat*}{3}
\lambda_{\eac}&:= (p(\fixedt)-A^TAd)_{\eac}~, &\quad \lambda_{(\eac)^C} &:= 0~, \\
\theta_{\mathcal{E}^C} &:= ( p(\fixedt)-A^TAd)_{\mathcal{E}^C}~, & \theta_{\mathcal{E}}&:= 0 ~.
\end{alignat*}
Then \eqref{directions:eq_KKT_eq_2} and \eqref{directions:eq_KKT_ineq_2} are satisfied by definition. \\
As there exists a  \( \delta > 0 \) such that  \( u(t) \in U_t \)  for all \( [\fixedt-\delta,\fixedt] \), we conclude
\begin{equation}\label{directions:eq_p_rewrite_2}
p(t)= p(\fixedt) + \frac{\fixedt-t}{t} (p(\fixedt)-A^TAd) = \frac{1}{t} A^T(f-Au(t)) \in \partial \| u(t) \|_1 \quad \forall t \in [\fixedt-\delta,\fixedt]~.
\end{equation}
Use this observation to first prove the multiplier equation \eqref{directions:eq_KKT_mult}, then the feasibility condition \eqref{directions:eq_KKT_eq_1}, and finally the equations \eqref{directions:eq_KKT_ineq_1}, \eqref{directions:eq_KKT_ineq_3}, and \eqref{directions:eq_KKT_ineq_4} concerning \( \lambda \). \\

To prove \eqref{directions:eq_KKT_mult}, we need to show \( (A^T A d )_\mathcal{A} = p(\fixedt)_\mathcal{A} \).  Since \( u(t) \) is continuous, 
\begin{equation*} 
\sgn(u(t)_i) = \sgn(u(\fixedt)_i) 
 \end{equation*}
for all  \( t \in [\fixedt-\delta, \fixedt ] \) and \( i \in \mathcal{A} \). Therefore \( p(t)_i = p(\fixedt)_i \), which together with \eqref{directions:eq_p_rewrite_2} yields \( (A^T A d)_i = p(\fixedt)_i \). \\
From  Lemma \ref{notation:lem_basic_facts} and the continuity of \( p(t) \), it follows that \( \mathcal{A}(u(t)) \subseteq \mathcal{E}(t) \subseteq \mathcal{E}(\fixedt) = \mathcal{E}\). Thus \( \supp(d) \subseteq \mathcal{E} \), which proves \eqref{directions:eq_KKT_eq_1}. \\
 
To conclude, we prove \eqref{directions:eq_KKT_ineq_1}, \eqref{directions:eq_KKT_ineq_3}, and \eqref{directions:eq_KKT_ineq_4}. For this, let \( {i \in \eac} \). First, assume that \( d_i \not = 0 \). Since  \( u(\fixedt)_i = 0 \), it follows that 
\begin{equation*}
p(t)_i = \sgn(u(t)_i) = \sgn( d_i ) \quad \forall t \in [\fixedt-\delta,\fixedt]~.
\end{equation*}
Thus \( p(t)_i \) is constant, and \eqref{directions:eq_p_rewrite_2} yields 
\( 
\lambda_i = (p(\fixedt)- A^TAd)_i = 0 
\). \\
Second, assume that \( d_i = 0 \). Then \eqref{directions:eq_KKT_ineq_1} and \eqref{directions:eq_KKT_ineq_4} follow immediately, and 
\begin{equation*}
p(t)_i \overset{\eqref{directions:eq_p_rewrite_2}}{=} p(\fixedt)_i + \frac{\fixedt-t}{t} \lambda_i ~.
\end{equation*}
If \eqref{directions:eq_KKT_ineq_3} were violated, then for all \( t \in [\fixedt-\delta,\fixedt) \)  we would have
\begin{equation*}
|p(t)_i|= |p(\fixedt)_i + \frac{\fixedt-t}{t} \lambda_i | > | p(\fixedt)_i |=1 ~,
\end{equation*}
which would contradict \( p(t) \in \partial \| u(t) \|_1 \). 
\end{proof}

\begin{remark}\label{directions:rem_explicit_t}
To implement the generalized homotopy method, we need an explicit expression for the maximal step size \( \delta=\delta(\fixedt,u(\fixedt),d) > 0 \). For this, define \( s_{\mathcal{A}},s_{\mathcal{E}\backslash \mathcal{A} },  s_{\mathcal{E}^C},s \in \mathbb{R} \) as
\begin{alignat*}{2}
s_\mathcal{A} 						&:= \max_{ i \in \mathcal{A} } \{ \nu_i \colon d_i \not = 0 \text{ and } \nu_i < \fixedt \}				  \text{ with }   \nu_i 		 := \frac{ u(\fixedt)_i + \fixedt d_i}{d_i} ~, \\
s_{\mathcal{E}\backslash \mathcal{A} } &:= \max_{ i \in \mathcal{E} \backslash \mathcal{A} } \{ \frac{|\lambda_i|}{|\lambda_i|+2} \fixedt \} ~,  	 			  		 						\\
s_{\mathcal{E}^C} 					&:= \max_{i \in \mathcal{E}^C } \{ \mu_i^\pm \colon \mu_i^\pm < \fixedt \}					          \text{ with }   \mu_i^{\pm}   := \begin{cases} \begin{tabular}{ll}
																																			 \( \fixedt \frac{ (p(\fixedt) - A^T A d )_i }{ \pm 1 - ( A^T A d )_i } \) & if \( \pm 1 \not = (A^TAd)_i \) \\
																																			 \( 0 \)													& else 
																																			 \end{tabular} \end{cases} \hspace{-3ex},\\
s								&:= \max \{ s_\mathcal{A}, s_{\mathcal{E}\backslash \mathcal{A} }, s_{\mathcal{E}^C}, 0 \}	~.																												\end{alignat*}
Then, the maximal step size is given by \( \delta= \fixedt-s  \). This follows directly from the preceding proof.
\end{remark}

\begin{corollary}\label{directions:cor_finitely_many}
There exist only finitely many sets of possible direction \( \Directionsft \).
\end{corollary}
\begin{proof}
The corollary essentially follows from the KKT conditions \eqref{directions:eq_KKT_mult}-\eqref{directions:eq_KKT_ineq_4} in the proof of Theorem \ref{directions:thm_characterization}. Recall that \( d \in \Directionsft \) is a possible direction if and only if there exist \( (\lambda,\theta) \in \mathbb{R}^N \times \mathbb{R}^N \) such that \eqref{directions:eq_KKT_mult}-\eqref{directions:eq_KKT_ineq_4} are satisfied. Since \( \theta_{\mathcal{E}^C} \) can be chosen freely, the KKT conditions depend only on \( \mathcal{E},\mathcal{A}, ~ \text{and} ~ p(\fixedt)_{\mathcal{E}}\in \{ \pm 1 \}^{ |\mathcal{E}|} \). Therefore, the set \Directionsft~ depends only on  \( \mathcal{E},\mathcal{A}, \text{ and}~ p(\fixedt)_{\mathcal{E}} \), which attain only finitely many different values.
\end{proof}

\section{The Generalized Homotopy Method}\label{section:hom_alg}

The characterization of the set of possible directions \( \Directionsft \) in Theorem \ref{directions:thm_characterization} directly yields a meta approach to compute a solution path: Start by choosing \( t^0 \) large enough to ensure that \( u(t^0)= 0 \) is a solution, i.e., 
\( t^0= \| A^T f \|_\infty \).
Compute a direction \( d^1 \in \Directions (t^0,u(t^0) ) \) and continue along the path \( t\mapsto \left( t, u(t^0)+(t^0-t) d^1 \right) \in \mathbb{R} \times \mathbb{R}^N \) as long as \( u(t) \in U_t \). Then, compute a new direction and repeat. \\
In the case of non-uniqueness, this approach yields a family of algorithms, as it needs to be combined with a rule \( R \) to choose a specific \( d \) from a given set \( \Directions \) of potential directions, i.e., \( d= R(\Directions) \in \Directions \). The proof of the finite termination property \cite{Osborne2000,Efron2004,Tibshirani2013} only holds for some and not for all of these algorithms as illustrated in Proposition  \ref{hom_alg:prop_infinite_kinks} below. Thus for certain choice rules, the meta approach does not necessarily terminate after finitely many steps.  

\begin{proposition}\label{hom_alg:prop_infinite_kinks}
There exists a choice rule \( R \), which, combined with the meta approach outlined above, yields a piecewise linear and continuous solution path \( {t \mapsto u(t)} \) with infinitely many kinks for certain \( A \in \mathbb{R}^{m \times N} \) and \( f \in \mathbb{R}^{m} \).
\end{proposition}
\begin{proof}
Let 
\begin{equation}\label{hom_alg:eq_infinite_kinks_matrix}
A=\begin{bmatrix} 1 & 1 & 1 & 0 \\ 0 & 0 & 0 & 1 \end{bmatrix}\in \mathbb{R}^{2\times 4} ~~\text{and}~~
  f= \begin{bmatrix} 2 \\ 1 \end{bmatrix} \in \mathbb{R}^2 ~.
\end{equation}
 Then \( u(t)= (u_1(t),u_2(t),u_3(t),u_4(t))^T \in U_t \) if and only if 
\begin{equation}\label{alg_hom:prop_proof_1}
\begin{cases}\begin{tabular}{llll}
 \( u_1=u_2=u_3= 0 \), & & \( u_4=0 \) & if \( t \geq 2 \) \\
 \( u_1,u_2,u_3\geq 0 \), & \( u_1+u_2+u_3=2-t \),&   \( u_4=0 \) & if \( t \in (1,2) \) \\
  \( u_1,u_2,u_3\geq 0 \), & \( u_1+u_2+u_3=2-t \),&   \( u_4=1-t \) & if \( t \in [0,1] \) ~.\\
\end{tabular}
\end{cases}
\end{equation}
\begin{figure}
\begin{subfigure}{0.49\textwidth}
\centering{
\includegraphics[width=\textwidth,keepaspectratio=true]{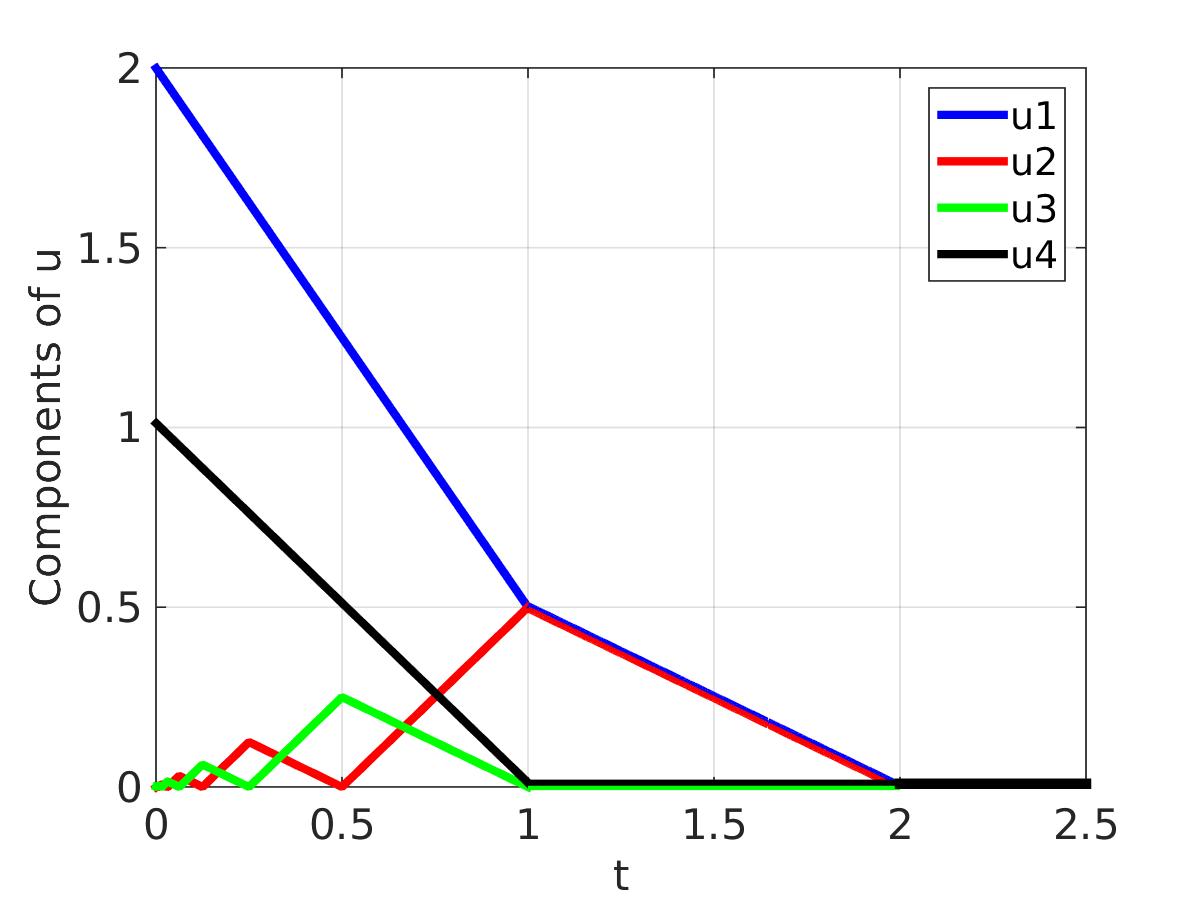}
\subcaption{Components of \( u(t) \)}}
\end{subfigure}
\begin{subfigure}{0.49\textwidth}
\centering{
\includegraphics[width=\textwidth,keepaspectratio=true]{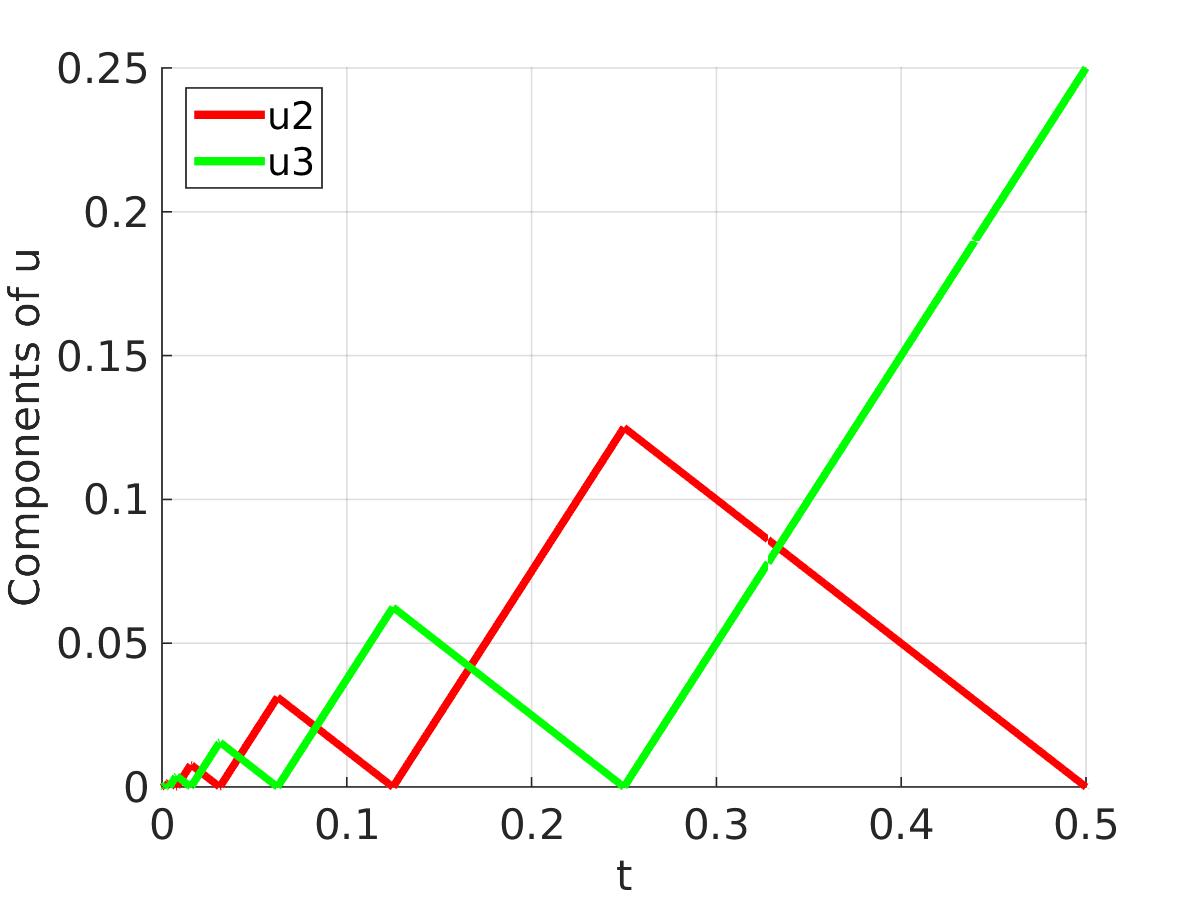}
\subcaption{\( u_2(t) \) and \( u_3(t) \) on \( [0,1/2] \)}}
\end{subfigure}
 \vspace{2ex}\\ 
 \caption{For \( A \) and \( f \) as in \eqref{hom_alg:eq_infinite_kinks_matrix}, we display a piecewise linear and continuous solution path with infinitely many kinks.}
 \label{alg_hom:fig_oscillating_path}
\end{figure}

Thus already at the fist kink \( t=2 \) there are multiple permissible directions, and it is not a priori clear which of them to choose. The choice \( d = \begin{bmatrix} \half & \half & 0 & 0 \end{bmatrix}^T \) is permissible, yielding \( u_1(t)=u_2(t)= \frac{2-t}{2} \), 
\( u_3(t)= u_4(t) = 0 \).  At \( t=1 \), a change of direction is necessary to prevent that \( u_4 \) violates \eqref{alg_hom:prop_proof_1}. A new permissible direction is \( d= \begin{bmatrix} \frac{3}{2} & -1 & \half & 1 \end{bmatrix}^T \). Now \( u_2(t) \) decreases and hits \( 0 \) at \( t=2 \), so again a change of direction is required; 
\( d= \begin{bmatrix} \frac{3}{2} & \half & -1 & 1 \end{bmatrix}^T \) is permissible. Continuing in this fashion and alternating between \( d= \begin{bmatrix} \frac{3}{2} & -1 & \half & 1 \end{bmatrix}^T \) and \( d= \begin{bmatrix} \frac{3}{2} & \half & -1 & 1 \end{bmatrix}^T \), one obtains kinks at  \( t=2^{-k} \) for every \( k \in \mathbb{N} \). It is easy to check that the three different directions chosen really correspond to different sets \( \Directions \), so we are following a choice rule \( R \). \\
The resulting solution path is displayed in Figure \ref{alg_hom:fig_oscillating_path}.
\end{proof}

To the best of our knowledge, this phenomenon was not discussed in any previous work dealing with the Lasso, nor have any specific choice rules been studied which avoid it. We propose to always choose the direction \( d^{j+1} \in \Directions(t^j,u(t^j)) \) with minimal \( \ell_2\)-norm, which yields the generalized homotopy method (Algorithm \ref{alg_hom:algorithm}).
This approach is computationally feasible. For example, by first computing any \( \tilde{d} \in \Directions(t^j,u(t^j)) \), it can be formulated as
\begin{equation}\label{alg_hom:eq_computation_minimal_l2_direction}
\begin{aligned}
d^{j+1} \in \argmin_{d\in \mathbb{R}^N} \| d \|_2^2 \quad \text{s.t.} ~ &Ad = A \tilde{d},~ d_{\Ej^C} = 0, \\
 &d_i p(t^j)_i \geq 0 ~ \forall i \in \Ej \backslash \Aj ~.
\end{aligned}
\end{equation} 
In the most common scenarios (see Lemmas \ref{alg_hom_other:lem_our_leaving}, \ref{alg_hom_other:lem_our_subsign} and \ref{alg_hom_other:lem_our_hitting}), the computation of \( d^{j+1} \) is simpler than \eqref{alg_hom:eq_computation_minimal_l2_direction}.

\begin{algorithm}
\caption{Generalized Homotopy Method}\label{alg_hom:algorithm}
\begin{algorithmic}[1]
\State {\textbf{Input:} data \( f \in \mathbb{R}^m\), matrix \( A\in \mathbb{R}^{m\times N} \) } 
\State {\textbf{Output:} number of steps \( K \), sequence \( t^0,\hdots,t^K \) of regularization parameters, sequence \( u(t^0),\hdots,u(t^K) \) of solutions}
\State {\textbf{Initialization:} Set \( t^0=\| A^T f \|_\infty~\text{and}~ u(t^0)=0 \).}
\For { \( j=0,1,\hdots \)} 
\If { \( t^j= 0 \) }
\State { Break }
\EndIf
\State { Compute \( r^j= f - A u(t^j),~ \mathcal{E}=\Ej,\text{ and } \mathcal{A}=\mathcal{A}(u(t^j)) \).}
\State { Set \begin{equation*}
\Directions= \argmin_d \| A d - \frac{1}{t^j} r^j \|_2^2 \quad \text{s.t.}~ d_{\mathcal{E}^C}=0, ~ d_i ~ p(t^j)_i\geq 0 ~ \forall i \in \mathcal{E}\backslash \mathcal{A}~,
\end{equation*}
\indent and compute 
\(
d^{j+1}=\argmin_{d\in \Directions} \| d \|_2^2 ~.
\)
}
\State { Using Remark \ref{directions:rem_explicit_t}, find the minimal \( t^{j+1} \geq 0 \) s.t. \( u(t)\in U_t \) for all \( t\in  [t^{j+1},t^j] \).}
\EndFor
\end{algorithmic}
\end{algorithm}

Let \( t^0=\| A^T f \|_\infty > t^1 > \hdots > t^K = 0  \) and \( u(t^0),u(t^1),\hdots, u(t^K) \) be the outputs of Algorithm \ref{alg_hom:algorithm}. The path \( u \colon \mathbb{R}_{\geq 0 } \rightarrow \mathbb{R}^N, ~ t \mapsto u(t) \) is then defined via linear interpolation by
\begin{equation}\label{alg_hom:eq_def_u}
u(t)= 
\begin{cases}
\begin{tabular}{ll}
\( 0 \) &if  \( t\geq t^0 \)\\
\( \frac{t-t^k}{t^{k-1}-t^k} u(t^{k-1})+ \frac{t^{k-1}-t}{t^{k-1}-t^k} u(t^k) \) & if \( t\in [t^k,t^{k-1}) \)
\end{tabular}
\end{cases}~.
\end{equation}
The following theorem, the main result of this paper, shows that \( u(t) \) is indeed a solution path.
\nopagebreak{
\begin{theorem}\label{alg_hom:thm_main}
The generalized homotopy method (Algorithm \ref{alg_hom:algorithm}) terminates after finitely many iterations. Furthermore,  \( u\colon \mathbb{R}_{\geq 0 } \rightarrow \mathbb{R}^N \) as in \eqref{alg_hom:eq_def_u} is piecewise linear, continuous, and satisfies
\begin{equation*}
 u(t)\in U_t =\argmin_{u\in \mathbb{R}^N} \half \| Au- f\|_2^2 + t ~\|u \|_1 
 \end{equation*}
 for all \( t > 0 \) as well as \( u(0) \in U_0 \).
\end{theorem}
}~\par

To prove Theorem \ref{alg_hom:thm_main}, we need the following lemma, which describes the dependence of the solution sets \( U_t \) on \( t \).
\begin{lemma}\label{alg_hom:lem_set_valued_linearity}
Let \( p(t) \) be the subgradient as in  Lemma \ref{notation:lem_basic_facts}. For every \( \mathcal{E}\subseteq [N] \) and every \( s \in \{ \pm 1 \}^{|\mathcal{E}|} \) the set \begin{equation*}
I_{\mathcal{E},s} = \{ 0 <  t \leq \| A^T  f \|_{\infty} \colon \mathcal{E}= \argmax_{i\in [N]} |p(t)_i |,~ p(t)_\mathcal{E}= s \}
\end{equation*}
is an interval. If \(a,b \) with \( a<b \) lie in the same  \( I_{\mathcal{E},s} \), and \(u(a)\in U_a \) as well as \( u(b)\in U_b \), then for every \( t \in [a,b] \) the linear interpolation
\begin{equation*}
u(t)=\frac{b-t}{b-a} u(a) + \frac{t-a}{b-a} u(b)
\end{equation*}
satisfies \( u(t) \in U_t \).
\end{lemma}
\begin{proof}
Fix \( \mathcal{E}\subseteq [N], s \in \{ \pm 1 \}^{|\mathcal{E}|} \), \( a,b\in I_{\mathcal{E},s} \), and \( u(a) \in U_a \) as well as \( u(b) \in U_b \). 
Let \( p(a) \) and \( p(b) \) be the subgradients at \( u(a) \) and \( u(b) \) as defined in  \eqref{notation:eq_the_subgradient}. Then in order to prove that \( u(t) \in U_t \), we have to show that 
\begin{equation*}
 \partial \| u(t) \|_1 \ni \tilde{p}(t):= \frac{1}{t} A^T( f- Au(t) ) = \frac{b-t}{b-a} \frac{a}{t} p(a) + \frac{t-a}{b-a} \frac{b}{t} p(b)~.
\end{equation*}
The last equality shows that \( \tilde{p}(t) \) is a convex combination of \( p(a) \) and \( p(b) \) for every \( t \in [a,b] \). Thus \( \| \tilde{p}(t)_{\mathcal{E}^C} \|_\infty < 1 \) and \( \tilde{p}(t)_\mathcal{E}= p(a)_\mathcal{E}
=p(b)_\mathcal{E}=s \). In particular, \( { \| \tilde{p}(t) \|_\infty \leq 1 } \). To show that \( \tilde{p}(t) \in \partial \| u(t) \|_1 \), it hence suffices to prove that \( \| u(t) \|_1 \leq \langle \tilde{p}(t),u(t) \rangle \). Since \( \mathcal{A}(u(a))\cup \mathcal{A}(u(b))\subseteq \mathcal{E} \), we have that
\begin{align*}
\langle \tilde{p}(t),u(t) \rangle &=  \frac{b-t}{b-a} \langle \tilde{p}(t),u(a) \rangle + 
						\frac{t-a}{b-a} \langle \tilde{p}(t),u(b) \rangle \\
					&= \frac{b-t}{b-a} \langle \tilde{p}(t)_\mathcal{E},u(a)_\mathcal{E} \rangle + 
						\frac{t-a}{b-a} \langle \tilde{p}(t)_\mathcal{E},u(b)_\mathcal{E} \rangle \\
					&= \frac{b-t}{b-a} \langle p(a)_\mathcal{E},u(a)_\mathcal{E} \rangle + 
						\frac{t-a}{b-a} \langle p(b)_\mathcal{E},u(b)_\mathcal{E} \rangle \\
					&= \frac{b-t}{b-a} \| u(a) \|_1 + \frac{t-a}{b-a} \| u(b) \|_1 \\
					&\geq \| u(t) \|_1 ~.
\end{align*}
This concludes the proof of the second statement. In particular, \( \tilde{p}(t) \) coincides with the subgradient \( p(t) \) as in \eqref{notation:eq_the_subgradient}. \\
The first statement follows from \( u(t) \in U_t \), \( \| p(t)_{\mathcal{E}^C} \|_\infty < 1 \), and \( p(t)_\mathcal{E} = s \).
\end{proof}

\begin{proof}[Proof of Theorem \ref{alg_hom:thm_main}]
Recalling Theorem \ref{directions:thm_characterization}, \( d^{j+1} \) is indeed a feasible direction at \( (t^j,u(t^j)) \) provided that \( u(t^j)\in U_{t^j} \). 
To see this, observe that 
\begin{equation*}
\{ t \in [0,t^{j-1}]\colon u(t)=u(t^{j-1})+(t^{j-1}-t) d^{j} \in U_t \}
\end{equation*}
is closed.  Alternatively, one can also use explicit computation of the maximal \( \delta>0 \) in Remark \ref{directions:rem_explicit_t}.  \\

We show the finite termination property by contradiction, assuming that the algorithm does not terminate after finitely many steps.
By Lemma \ref{alg_hom:lem_set_valued_linearity} and the monotonicity of the \( t^j \), there exists a set \( \mathcal{E} \subseteq [N] \), a vector \( s \in \{ \pm 1 \}^{\mathcal{E} } \), and a \( j_0 \in \mathbb{N} \)  such that \( t^j \in I_{\mathcal{E},s} \) for all \( j \geq j_0 \). \\
We will now show that 
\begin{equation}\label{alg_hom:eq_proof_main_thm}
\| d^{j+1} \|_2 < \| d^{j+2} \|_2 \qquad \forall ~ j \geq j_0 ~.
\end{equation}
Since \( t^j,t^{j+2} \in I_{\mathcal{E},s} \), there exists, again by Lemma \ref{alg_hom:lem_set_valued_linearity}, a \( \bar{d}\in \Directions(t^j,u(t^j)) \) such that \begin{equation*}
u(t^{j+2})=u(t^j)+(t^j-t^{j+2}) \bar{d}~.
\end{equation*}
By construction, 
\begin{equation*}
u(t^{j+2})=u(t^j)+(t^j-t^{j+1})d^{j+1} + (t^{j+1}-t^{j+2}) d^{j+2}~.
\end{equation*}
It follows that \begin{equation*}
\bar{d}= \frac{t^j-t^{j+1}}{t^j-t^{j+2}} d^{j+1} + \frac{t^{j+1}-t^{j+2}}{t^j-t^{j+2}} d^{j+2}
\end{equation*}
is a convex combination of \( d^{j+1} \) and \( d^{j+2} \). If \( \| d^{j+2} \|_2 \leq \| d^{j+1} \|_2 \), we would have  \( \| \bar{d} \|_2 \leq \| d^{j+1} \|_2 \) by convexity. But, since \( d^{j+1} \) is the unique direction with minimal \( \ell_2\)-norm (by strict convexity), this yields \( \bar{d}=d^{j+1} \), which is a contradiction to \( d^{j+2} \not = d^{j+1} \). This shows \eqref{alg_hom:eq_proof_main_thm}. \\
By Corollary \ref{directions:cor_finitely_many}, there exist only finitely many sets of possible directions \( \Directions(t^j,u(t^j)) \). Since \( d^{j+1} \) is uniquely determined for each \( \Directions(t^j,u(t^j) )\), the set \( \{ d^{j+1} \colon j\in \mathbb{N} \} \) is also finite, a  contradiction to \eqref{alg_hom:eq_proof_main_thm}.
\end{proof}

\section{Relation to previous work}\label{section:relation_alg}\label{section:hom_alg_other}
In this section, we compare the generalized homotopy method with previous homotopy algorithms \cite{Osborne2000,Efron2004,Loris2008,Tibshirani2013} and the adaptive inverse scale space method \cite{Burger2012}. 

\subsection{Standard Homotopy Method}\label{subsection:standard_homotopy}
At the core of te generalized homotopy method is a nonnegative least squares prolem to locally choose the direction. In contrast, previous works on the homotopy method \cite{Osborne2000,Efron2004,Tibshirani2013}  proposed to find a direction by solving a linear system. We will refer to the resulting algorithm as the standard homotopy method (Algorithm \ref{alg_hom_other:standard_algorithm}). Note that for reasons of better comparison to Algorithm \ref{alg_hom:algorithm}, we have slightly extended the method to also accept inputs where the one-at-a-time condition (see Definition \ref{alg_hom_other:def_one_at_a_time}) does not hold. A first step towards dealing with scenarios where the one-at-a-time condition fails is the homotopy method with looping, which was, based on ideas in \cite{Efron2004}, introduced in \cite{Loris2008}, and is summarized in Algorithm \ref{alg_hom_other:algorithm_looping}.

\begin{algorithm}[h]
\caption{Standard Homotopy Method \cite{Osborne2000,Efron2004}}\label{alg_hom_other:standard_algorithm}
\begin{algorithmic}[1]
\State {\textbf{Input:} data \( f \in \mathbb{R}^m\), matrix \( A\in \mathbb{R}^{m\times N} \) } 
\State {\textbf{Output:} number of steps \( K \), sequence \( t^0,\hdots,t^K \) of regularization parameters, sequence \( u(t^0),\hdots,u(t^K) \) of solutions}
\State {\textbf{Initialization:} Set \( t^0=\| A^T f \|_\infty~\text{and}~ u(t^0)=0 \).}
\For { \( j=0,1,\hdots \)} 
\If { \( t^j= 0 \) }
\State { Break }
\EndIf
\State { Set \( \Leav^j= \mathcal{A}(u(t^{j-1}))\backslash \mathcal{A}(u(t^j)) \) and \( \StandardS^j = \Ej \backslash \Leav^j \).}
\State { Compute 
\begin{equation}\label{alg_hom_other:eq_standard_direction}
d^{j+1}_{\StandardS^j} = \left (A_{\StandardS^j}^T A_{\StandardS^j} \right)^\dagger p(t^j)_{\StandardS^j}  \quad \text{and set} \quad d^{j+1}_{\left( \StandardS^j \right)^C } = 0 ~.
\end{equation}
}

\State { Find the minimal \( t^{j+1} \geq 0 \) such that \( u(t) \) solves \eqref{directions:eq_variational_problem} on \( [t^{j+1},t^j] \).}
\If { \( t^{j+1}=t^j \) }
\State { \textbf{Error} ``Algorithm failed to produce a solution path.'' }
\EndIf
\EndFor
\end{algorithmic}
\end{algorithm}

\begin{definition}[\cite{Efron2004}]\label{alg_hom_other:def_one_at_a_time}
Let  \( t^0,\hdots, t^K \) and \( u(t^0),\hdots, u(t^K) \) be the output produced by Algorithm \ref{alg_hom_other:standard_algorithm}. An index \( i \in \Leav^j := \mathcal{A}(u(t^{j-1}))\backslash \mathcal{A}(u(t^j)) \) is called a leaving coordinate, and an index \( i \in \Hit^j:=\Ej \backslash \Ejminus \) is called a hitting coordinate. \\
We say the \emph{one-at-a-time} condition is satisfied, if for every \( 0\leq j \leq K \) such that \( t^j > 0 \) we have that
\begin{equation}\label{alg_hom_other:eq_one_at_a_time}
| \Hit^j \dot{\cup} \Leav^j |\leq 1 ~.
\end{equation}
\end{definition}

\begin{theorem}[\cite{Efron2004}]\label{alg_hom_other:thm_efron}
Assume that the one-at-a-time condition holds and that \( A_{\Ej } \) is injective at every iteration. Then, the standard homotopy algorithm computes the unique solution path in finitely many steps. 
\end{theorem}
 
Our next result shows that under the same assumptions, the standard and generalized homotopy methods agree.
 
\begin{theorem}\label{alg_hom_other:thm_coincide}
Assume that the  one-at-a-time condition holds and that \( A_\Ej \) is injective at every iteration. Then the outputs of the standard homotopy method \cite{Osborne2000, Efron2004} and the generalized homotopy method  (Algorithm \ref{alg_hom:algorithm}) coincide.\\
\end{theorem}

\begin{remark}
As the injectivity assumption guarantees the uniqueness of the solution path, Theorem \ref{alg_hom_other:thm_coincide} directly follows from Theorem \ref{alg_hom:thm_main} and Theorem \ref{alg_hom_other:thm_efron}. Nevertheless, we provide a self-contained proof. This also yields an alternative proof of Theorem \ref{alg_hom_other:thm_efron}.
\end{remark}

To prove Theorem \ref{alg_hom_other:thm_coincide}, we first show Proposition \ref{alg_hom:proposition_linear_system}, which states that \( d^{j+1} \) as in Algorithm \ref{alg_hom:algorithm} has the form \eqref{alg_hom_other:eq_standard_direction} for some, a-priori unknown, set \( \mathcal{S}^j \). The following three lemmas then show that the support set agrees with \( \StandardS^j \) as in Algorithm \ref{alg_hom_other:standard_algorithm}. They correspond to the three different cases in the proof of Theorem \ref{directions:thm_characterization}.

\begin{proposition}\label{alg_hom:proposition_linear_system}
Let \( 0 \leq j \leq K-1 \).  Let \( \mathcal{S}^{j} := \Aj \cup \supp(d^{j+1}) \), where \( d^{j+1} \) is as in Algorithm \ref{alg_hom:algorithm}. Then
\( \Aj \subseteq \mathcal{S}^j \subseteq \Ej \cap \mathcal{E}(t^{j+1})\),
\begin{equation*}
d^{j+1}_{\mathcal{S}^j} = (A_{\mathcal{S}^j}^T A_{\mathcal{S}^j} )^\dagger p(t^j)_{\mathcal{S}^j}~, \quad \text{and} \quad d^{j+1}_{(\mathcal{S}^j)^C} = 0 ~.
\end{equation*}
\end{proposition}
\noindent In light of Proposition \ref{alg_hom:proposition_linear_system}, the standard homotopy method makes the educated guess \( {\mathcal{S}^j= \Ej \backslash \Leav^j } \), which is always correct under the assumptions in Theorem \ref{alg_hom_other:thm_efron}. 
\begin{proof}
Throughout the proof, we write \( {\mathcal{E}=\Ej} , ~{\mathcal{A}=\mathcal{A}(u(t^j))}\), \( \mathcal{S}= \mathcal{S}^j \) and \( \Directions=\Directions(t^j,u(t^j)) \). Set \begin{equation*}
\beta_\mathcal{S}= \left ( A_\mathcal{S}^T A_\mathcal{S} \right)^\dagger p(t^j)_\mathcal{S} 
			=\left ( A_\mathcal{S}^T A_\mathcal{S} \right)^\dagger A_\mathcal{S}^T \frac{r^j}{t^j} \quad \text{ and } \quad 
			\beta_{\mathcal{S}^C} = 0 .
\end{equation*}
We  have to show \( \beta = d^{j+1} \). Since the sign-constraint is only imposed on indices in \( \mathcal{E}\backslash \mathcal{A} \), there exists an \( \epsilon \in (0,1) \) such that \( \beta(\tau)=(1-\tau) d^{j+1} + \tau \beta \) is in the feasible set of the nonnegative least squares problem  \eqref{directions:eq_characterization}. Further, by the KKT conditions \eqref{directions:eq_KKT_mult}, \eqref{directions:eq_KKT_eq_2}, and \eqref{directions:eq_KKT_ineq_4}, we have
\begin{equation*}
A_\mathcal{S}^T A_\mathcal{S} d^{j+1}_\mathcal{S} = p(t^j)_\mathcal{S} = A_\mathcal{S}^T A_\mathcal{S} \beta_\mathcal{S} ~.
\end{equation*}
Thus \( Ad^{j+1}=A\beta \) holds, and \( \beta (\tau) \in \Directions \). By the definition of \( \beta \),  \( \| \beta \|_2 \leq \| d^{j+1} \|_2 \) holds, which yields \( \| \beta(\tau) \|_2 \leq \| d^{j+1} \|_2 \) for all \( \tau \in [0,\epsilon] \). Since \( \beta(\tau) \in \Directions \), this yields \( \beta(\tau)=d^{j+1} \). Recalling the definition of \( \beta(\tau) \), the equality \( \beta=d^{j+1} \) follows. \\

The inclusion \( \Aj \subset \mathcal{S}^j \) holds by the definition of \( \mathcal{S}^j \), and \( \mathcal{S}^j \subseteq \Ej \) folows from Theorem \ref{directions:thm_characterization}. To see that \( \mathcal{S}^j \subseteq \mathcal{E}(t^{j+1}) \), we distinguish two cases. If \( t^{j+1} > 0 \), then \begin{equation*}
p(t^{j+1})_{\mathcal{S}^j} \overset{\eqref{directions:eq_p_rewrite}}{=} p(t^j)_{\mathcal{S}^j} + \frac{t^j-t^{j+1}}{t^{j+1}} (p(t^j)_{\mathcal{S}^j} - A_{\mathcal{S}^j}^T A d^{j+1} ) = p(t^j)_{\mathcal{S}^j}~,
\end{equation*}
showing that \( \mathcal{S}^j \subseteq \mathcal{E}(t^{j+1}) \). Since \( \mathcal{E}(0) = [N] \), \( \mathcal{S}^j \subseteq \mathcal{E}(t^{j+1}) \) also holds if \( t^{j+1} = 0 \).
\end{proof}

The first lemma describes a case of  a leaving coordinate , i.e., that a coordinate \( u(t)_i \) becomes zero which was previously non-zero.
\begin{lemma}\label{alg_hom_other:lem_our_leaving}
Let \( \Hit^j = \emptyset \),  \( \Leav^j= \{ i \} \), \( \Ej = \Aj \dot{\cup} \{ i \} \) and \( d^{j+1} \) as in Algorithm \ref{alg_hom:algorithm}.
Then one has that 
\begin{equation*}
d^{j+1}_{\Ej \backslash\{i\} } = \left ( A^T_{\Ej \backslash\{i\} } A_{\Ej \backslash\{i\} } \right)^\dagger p(t^j)_{\Ej \backslash\{i\} } \quad \text{and} \quad d^{j+1}_{(\Ej \backslash\{i\})^C }=0~.
\end{equation*}
Furthermore, if \( A_\Ej \) is injective, then \( A_i^T A d^{j+1} \not = p(t^j)_i \) and the index \( i \) leaves the equicorrelation set \( \Et \).
\end{lemma} 
\begin{proof} We set \( \mathcal{S}^j \) as in  Proposition \ref{alg_hom:proposition_linear_system}. 
Since \( \Ej = \Aj \dot{\cup}  \{ i \} \), either \( \mathcal{S}^j=\Aj \) or \( \mathcal{S}^j=\Ej \).
From  \( u(t^{j-1} )_i \not = 0 \) (by the definition of \( \Leav^j \)) and \( u(t^j)_l \not = 0 \) for all \( l \in \Ej \backslash \{ i \} = \Aj  \) together with Proposition \ref{alg_hom:proposition_linear_system}, it follows that \( \mathcal{S}^{j-1}= \Ej \). Now \( \mathcal{S}^{j} \not = \Ej = \mathcal{S}^{j-1}  \), as otherwise, again by Proposition \ref{alg_hom:proposition_linear_system}, \( d^{j+1}=d^{j} \), which is a contradiction. \\
Now assume that \( A_\Ej \), and hence also \( A_\Ej^T A_\Ej \), is injective and that \( A_i^T A d^{j+1} = p(t^j)_i \). Then \begin{equation*}
A_\Ej^T A_\Ej d^{j+1}_\Ej = p(t^j)_\Ej = A_\Ej^T A_\Ej d^{j}_\Ej~,
\end{equation*}
and thus \( d^{j+1}=d^{j} \), which is again a contradiction.
\end{proof}

The second lemma deals with the case that \( \Hit^j = \emptyset \) and \( \Leav^j = \emptyset \). Under the additional assumption that \( \Ej = \Aj \dot{\cup} \{ i \} \), this implies that \( i \in [N] \) is in the equicorrelation set for both \( t=t^{j-1} \) and \( t=t^j \), but on the interval \( [t^j,t^{j-1} ] \) the \( i \)-th component \( p(t)_i \) changes from \( +1 \) to \( - 1\) or vice versa while \( u(t)_i \) remains zero.

\begin{lemma}\label{alg_hom_other:lem_our_subsign}
Assume that  \( \Ej =\Ejminus \) and that \( i \in \Ej  \) is an index such that \( \mathcal{A}(u(t^j))=\mathcal{A}(u(t^{j-1}))=\Ej \backslash \{ i \} \) and \( p(t^j)_i \not = p(t^{j-1} )_i \). Then 
\begin{equation*}
d^{j+1}_{\Ej }= \left ( A_{\Ej }^T A_{\Ej } \right)^\dagger p(t^j)_{\Ej } 
\quad \text{and} \quad d^{j+1}_{(\Ej )^C}=0~.
\end{equation*}
Furthermore, \( d^{j+1}_i p(t^j)_i > 0 \).
\end{lemma}
\begin{proof} 
Again, for \( \mathcal{S}^j \) as in Proposition \ref{alg_hom:proposition_linear_system}, one has that  either \( \mathcal{S}^j= \Ej \backslash \{ i \} \) or \( \mathcal{S}^j= \Ej  \).
If \( \mathcal{S}^{j-1} = \mathcal{E}(t^{j-1}) \), Proposition \ref{alg_hom:proposition_linear_system} would yield 
\begin{equation*}
A_{\Ejminus}^T A d = p(t^{j-1})_{\Ejminus}~,
\end{equation*}
implying that \( p(t^{j-1})_i = (A^T A d^j )_i \), and hence, with \eqref{directions:eq_p_rewrite}, a contradiction to the assumption \( p(t^{j+1})_i \not = p(t^j)_i~. \)
Thus, \( \mathcal{S}^{j-1} = \Ej \backslash \{ i \} \). Since \( d^{j+1} \not = d^{j} \) and \( p(t^j)_{\mathcal{A}(u(t^{j-1}))} = p(t^{j-1})_{\mathcal{A}(u(t^{j-1}))} \), it follows that \( \mathcal{S}^j \not = \mathcal{S}^{j-1} \), i.e., \( \mathcal{S}^j = \Ej  \), which proves the first part of the lemma. \\
Lastly, since \( d^{j+1}_i p(t^j)_i \geq 0 \) by Algorithm \ref{alg_hom:algorithm} and \( i \in \mathcal{S}^j \backslash \Aj \subseteq \supp(d^{j+1}) \), it follows that \( d^{j+1}_i p(t^j)_i > 0 \).
\end{proof}

The third lemma describes the case of a hitting coordinate, i.e., a coordinate which has to be included in the equicorrelation set at \( t=t^j \).
\begin{lemma}\label{alg_hom_other:lem_our_hitting}
Let \( \Hit^j= \{i\},~ \Leav^j = \emptyset, \) and \( \Ej = \Aj \dot{\cup}\{ i \} \). Then we have that
\begin{equation*}
d^{j+1}_{\Ej} = \left( A_{\Ej }^T A_{\Ej} \right)^\dagger p(t^j)_{\Ej }~, \quad d^{j+1}_{\Ej^C} = 0 ~,
\end{equation*}
and \( d^{j+1}_i p(t^j)_i > 0 \).
\end{lemma}
\begin{proof} By assumption \( \Ej = \mathcal{A}(u(t^j)) \dot{\cup} \{ i \} \). Then, either \( \mathcal{S}^j = \mathcal{A}(u(t^j)) \) or \( \mathcal{S}^j = \Ej \). To prove the lemma, it suffices to show that \( \mathcal{S}^j \not = \mathcal{A}(u(t^j)) \). \\
Since \( u(t^j)_l \not = 0 \) for all \( l \in \Aj \) and \( \Leav^j = \emptyset \), it follows that \( \mathcal{S}^{j-1} = \Aj \). If \( \mathcal{S}^j = \mathcal{A}(u(t^j)) \), then \( d^{j+1}=d^j \), which is a contradiction. The last inequality follows as in Lemma \ref{alg_hom_other:lem_our_subsign}.
\end{proof}
Lemma \ref{alg_hom_other:lem_our_hitting} was already proven in \cite[Lemma 5.4]{Efron2004} under the additional assumption that \( A_{\Ej} \) is injective. The most difficult part in their proof is to show that \( d^{j+1}_i \) agrees in sign with \( p(t^j)_i \), i.e., \( d^{j+1}_i p(t^j)_i \geq 0 \).  

\begin{proof}[Proof of Theorem \ref{alg_hom_other:thm_coincide}]
The result follows from the first parts of the Lemmas \ref{alg_hom_other:lem_our_leaving}, \ref{alg_hom_other:lem_our_subsign} and \ref{alg_hom_other:lem_our_hitting} if we show that the assumption \( |\Ej \backslash \Aj | = 1 \) is satisfied for every \( j=0,\hdots,K-1 \). Due to the one-at-a-time condition, it suffices to note that \( |\mathcal{E}(t)\backslash \mathcal{A}(u(t))|= 0 \) for every \( t\in (t^{j+1},t^j) \), which follows directly from the second parts of the Lemmas \ref{alg_hom_other:lem_our_leaving}, \ref{alg_hom_other:lem_our_subsign} and \ref{alg_hom_other:lem_our_hitting}.
\end{proof}

\begin{remark}
Our analysis shows that the injectivity of \( A_\Ej \) is only needed to show that  \( 1 = |\Ej \backslash \Aj |  \) at every iteration, and only in the scenario of Lemma \ref{alg_hom_other:lem_our_leaving}. Essentially, we have to exclude that a leaving index \( i \in \Leav^j =  \mathcal{A}(u(t^{j-1})) \backslash \Aj \) remains in the equicorrelation set, i.e.,   \( i \in \Et \) for all \( t \in (t^{j+1},t^j) \). As long as the solution path is unique, this would contradict Lemma \ref{alg_hom:lem_set_valued_linearity}. In the case of non-uniqueness, there may be additional kinks in the interior of one of the \( I_{\mathcal{E},s} \) (as defined in Lemma \ref{alg_hom:lem_set_valued_linearity}), where we do not know yet whether and how it can be excluded.
\end{remark}
\begin{figure}[!t]
\begin{subfigure}{0.49\textwidth}
 \includegraphics[width=\textwidth,keepaspectratio=true]{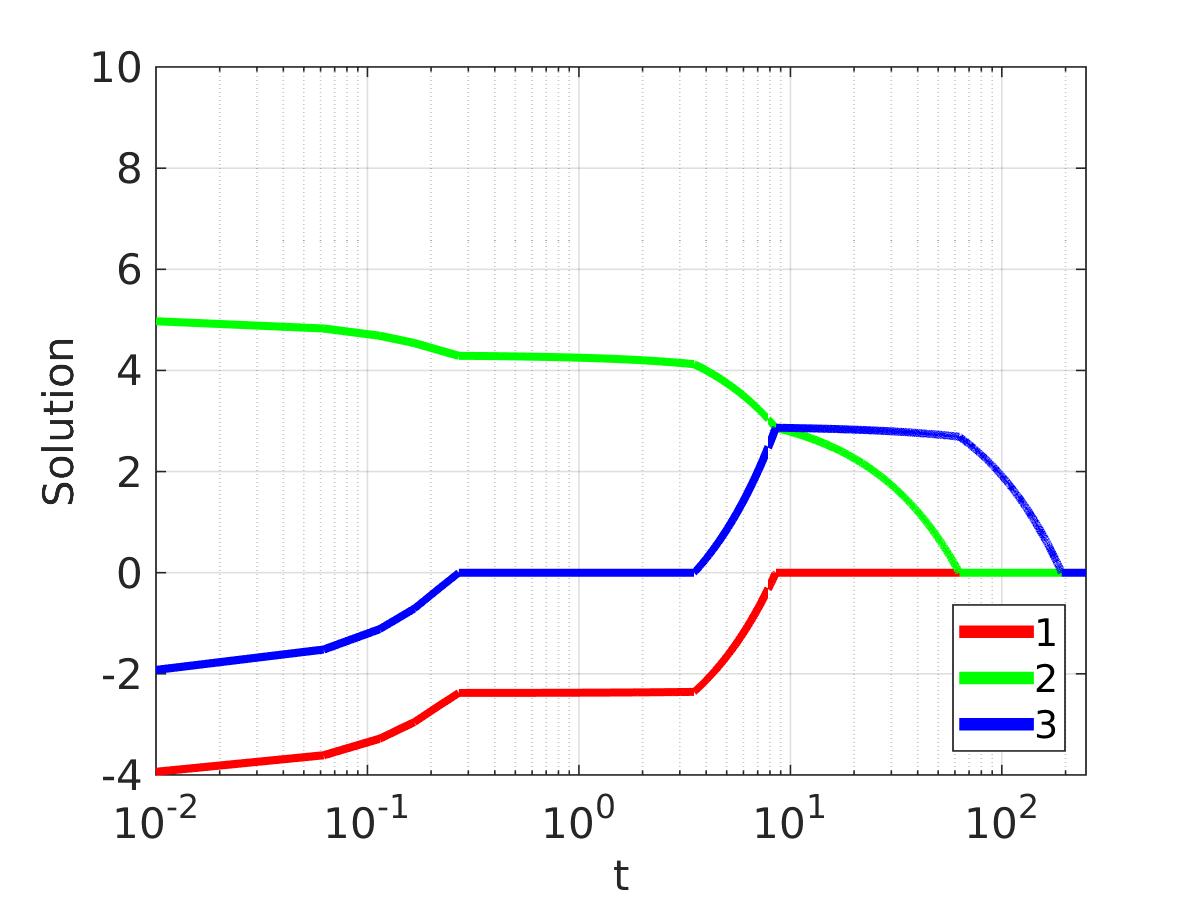} 
 \subcaption{Output of Algorithm \ref{alg_hom:algorithm}}
 \end{subfigure}
 \begin{subfigure}{0.49\textwidth}
\includegraphics[width=\textwidth,keepaspectratio=true]{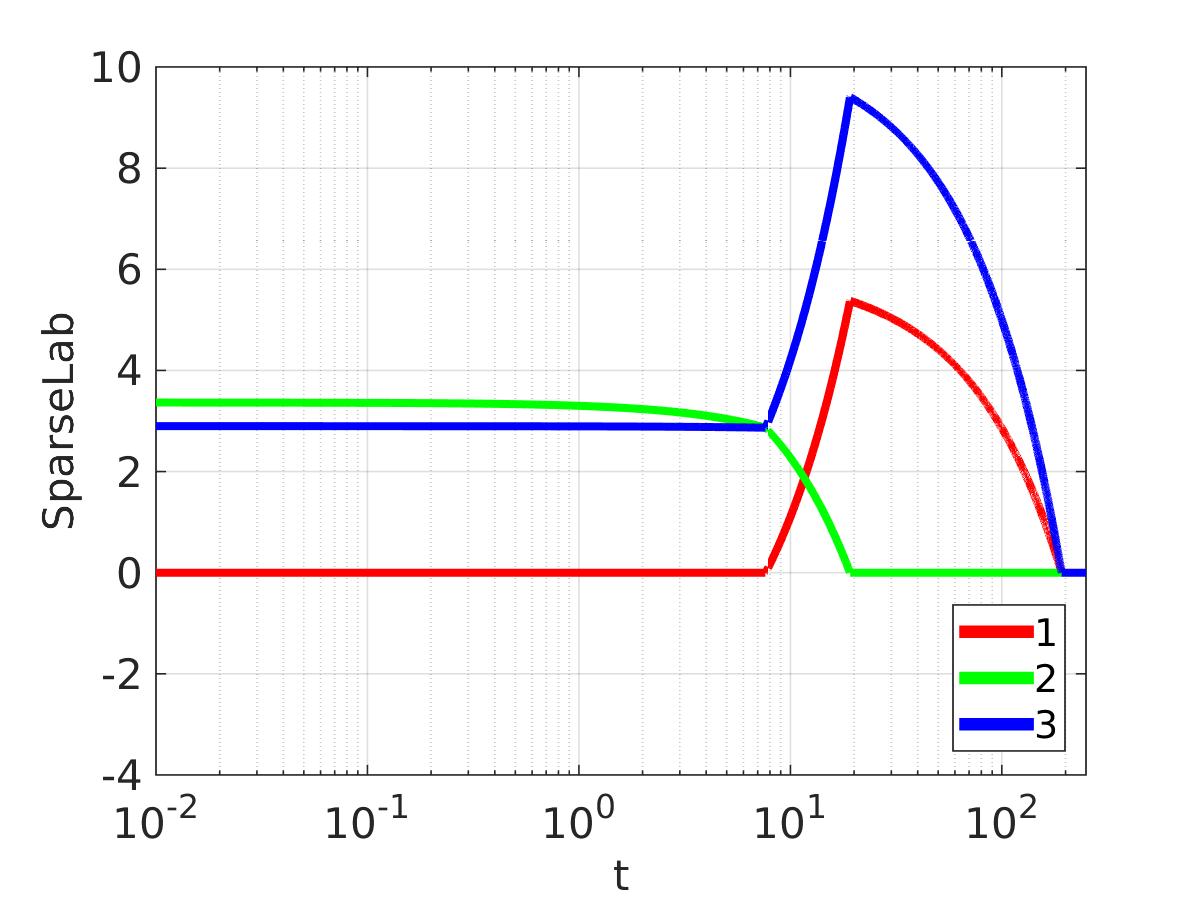} 
\subcaption{Output of SparseLab}
\end{subfigure}
\begin{subfigure}{0.5\textwidth}
\includegraphics[width=\textwidth,keepaspectratio=true]{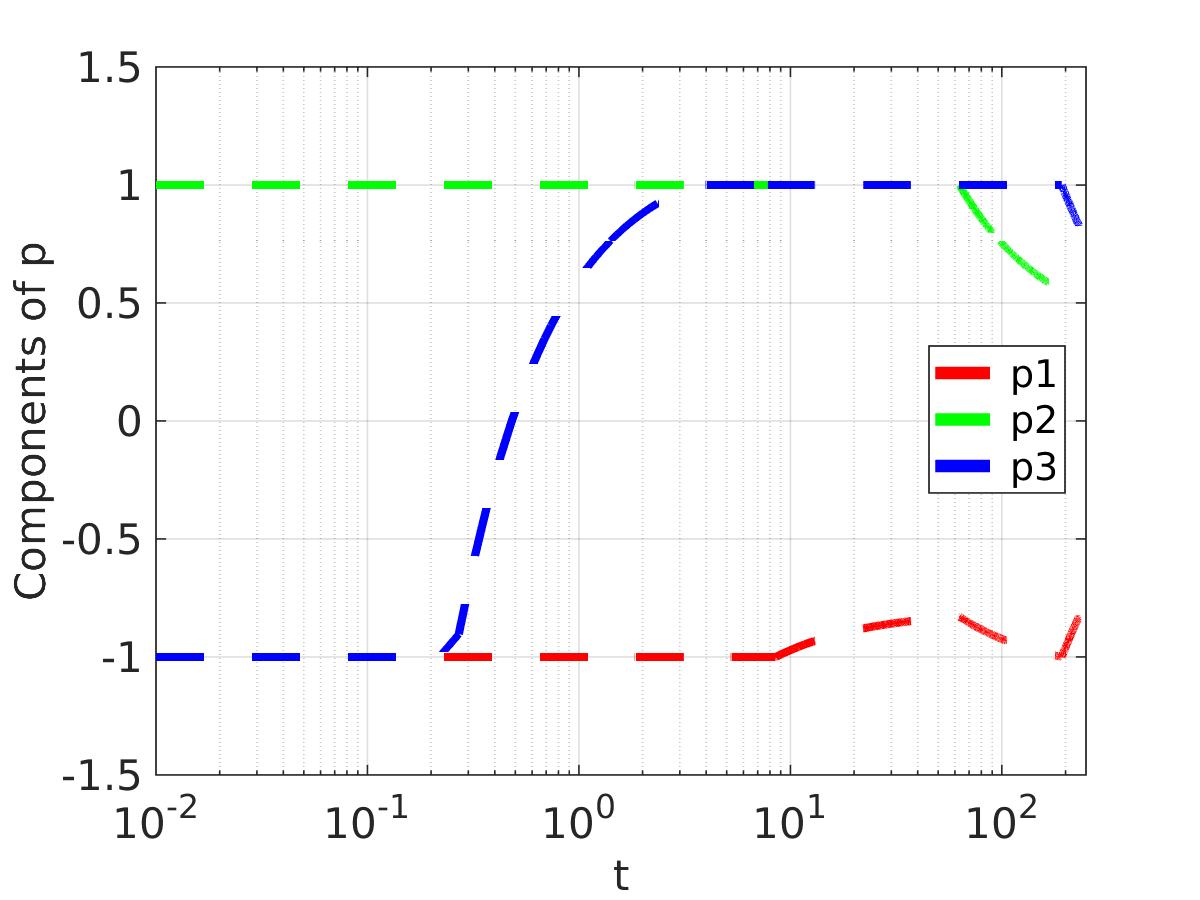} 
\subcaption{Components of \( p(t) \) }
\end{subfigure}
\caption{We display the behaviour of the generalized homotopy method and the output of SparseLab in an example where the one-at-a-time condition does not hold. On certain intervals, the output of SparseLab has a non-zero first component while the correct subgradient satisfies \( |p_1(t)|<1 \). Thus, SparseLab does not produce a solution path.}
\label{alg_hom_other:fig_loris}
\end{figure}

It was noted in \cite{Efron2004,Loris2008} that without the one-at-a-time condition, the standard homotopy method can encounter sign inconsistencies.
An example with \( A \in \mathbb{R}^{3 \times 3 } \) is given in \cite{Loris2008}, namely
 \begin{equation}\label{alg_hom_other:eq_loris_example}
 A= \begin{bmatrix}
 -3 & 4 & 4 \\ -5 & 1 & 4 \\ 5 & 1 & -4 
 \end{bmatrix}
 \quad \text{and} \quad 
 f= \begin{bmatrix}
 24 \\ 17 \\ -7 
 \end{bmatrix}~.
 \end{equation}
The outputs of the generalized homotopy method and SparseLab \cite{SparseLab2007}, one of the most popular implementations of the homotopy method \cite{Osborne2000,Efron2004},
 are displayed in Figure \ref{alg_hom_other:fig_loris}.  At the first kink \( t=192 \), the \( 3^{rd} \) component in the output produced by SparseLab enters the active set with the wrong sign. Therefore, in this example SparseLab is unable to produce a full solution path. As Algorithm \ref{alg_hom_other:standard_algorithm} has an additional feature to detect sign inconsistencies, it would exit at \( t = 192 \). \\
Notice that the matrix \( A \) is invertible, and thus the solution path is unique. The wrong solution path produced by SparseLab is solely the result of the missing one-at-a-time condition, and not of non-uniqueness. \\

In \cite{Loris2008}, based on ideas in \cite{Efron2004}, the following strategy was proposed:
Instead of choosing \( \StandardS^j = \Ej \backslash \Leav^j \) in Algorithm \ref{alg_hom_other:standard_algorithm}, loop over all sets \( \StandardS \subseteq [ N ] \) with \( \mathcal{A}(u(t^j)) \subseteq \StandardS \subseteq \Ej \), compute
\begin{equation*}
d_{\StandardS} = \left (A_{\StandardS}^T A_{\StandardS} \right)^\dagger p(t^j)_{\StandardS},  \quad \text{and set} \quad 
d_{ \StandardS ^C } = 0 ~.
\end{equation*}
Choose \( \StandardS^j = \StandardS \) as soon as \( t^{j+1}< t^j \), i.e., \( d \in \Directions(t^j,u(t^j)) \). We call this the homotopy method with looping (see Algorithm \ref{alg_hom_other:algorithm_looping}).\\ 

\begin{algorithm}
\caption{Homotopy Method With Looping \cite{Efron2004,Loris2008}}\label{alg_hom_other:algorithm_looping}
\begin{algorithmic}[1]
\State {\textbf{Input:} data \( f \in \mathbb{R}^m\), matrix \( A\in \mathbb{R}^{m\times N} \) } 
\State {\textbf{Output:} number of steps \( K \), sequence \( t^0,\hdots,t^K \) of regularization parameters, sequence \( u(t^0),\hdots,u(t^K) \) of solutions}
\State {\textbf{Initialization:} Set \( t^0=\| A^T f \|_\infty~\text{and}~ u(t^0)=0 \).}
\For { \( j=0,1,\hdots \)} 
\If { \( t^j= 0 \) }
\State { Break }
\EndIf
\For { \( S \subseteq [N] \) with \( \mathcal{A}(u(t^j)) \subseteq S \subseteq \mathcal{E}(u(t^j)) \)} 
\State { Compute 
\begin{equation*}
d^{j+1}_{\StandardS} = \left (A_{\StandardS}^T A_{\StandardS} \right)^\dagger p(t^j)_{\StandardS}  \quad \text{and set} \quad d^{j+1}_{\left( \StandardS \right)^C } = 0 ~.
\end{equation*}
}
\State { Find the minimal \( t^{j+1} \geq 0 \) such that \( u(t) \) solves \eqref{directions:eq_variational_problem} on \( [t^{j+1},t^j] \).}
\If { \( t^{j+1} < t^j \) }
\State { Break }
\EndIf
\EndFor
\EndFor
\end{algorithmic}
\end{algorithm}
From Proposition \ref{alg_hom:proposition_linear_system} it follows that the homotopy method with looping finds a direction at every iteration. This is, at least in the case of non-uniqueness, a non-trivial result. Notice that the homotopy method with looping does not necessarily compute the direction with minimal \( \ell_2 \) norm. Therefore, the finite termination of the algorithm is unclear. \\
Besides providing a theoretical foundation to the homotopy method with looping, the characterization of the set of possible directions (Theorem \ref{directions:thm_characterization})
 can also improve its performance. The loop over all sets \( \StandardS \subseteq [N] \) with \( \mathcal{A}(u(t^j)) \subseteq \StandardS \subseteq \Ej \) can be interpreted as a rudimentary active-set strategy to solve the nonnegative least squares problem \eqref{directions:eq_characterization}, even though this was not explicitly noted. As soon as \( | \Ej \backslash \mathcal{A}(u(t^j))| \) becomes large this methods becomes infeasible. Indeed, we would have to solve \( 2^{| \Ej \backslash \mathcal{A}(u(t^j))| } \) linear systems. Empirical tests show that small random Bernoulli matrices, for instance \( A \in \mathbb{R}^{20 \times 50} \), regularly yield \( | \Ej \backslash \mathcal{A}(u(t^j))| \geq 18 \). Nevertheless we consider their work \cite{Loris2008} an important step towards understanding the solution paths of \eqref{directions:eq_variational_problem} even when the one-at-a-time condition fails.

\begin{figure}[!t]
\begin{subfigure}{0.49\textwidth}
 \includegraphics[width=\textwidth,keepaspectratio=true]{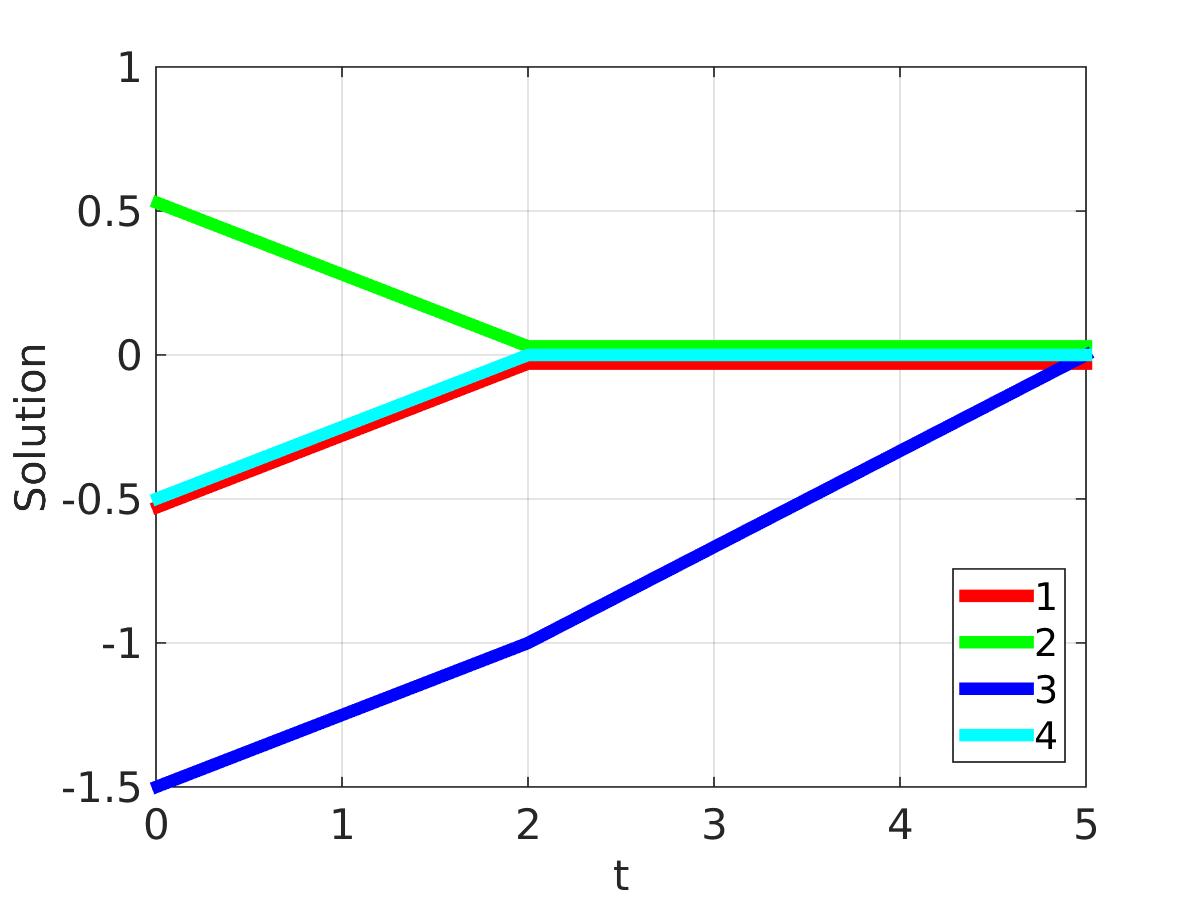} 
 \subcaption{Components of \( u(t) \)}
 \end{subfigure}
 \begin{subfigure}{0.49\textwidth}
\includegraphics[width=\textwidth,keepaspectratio=true]{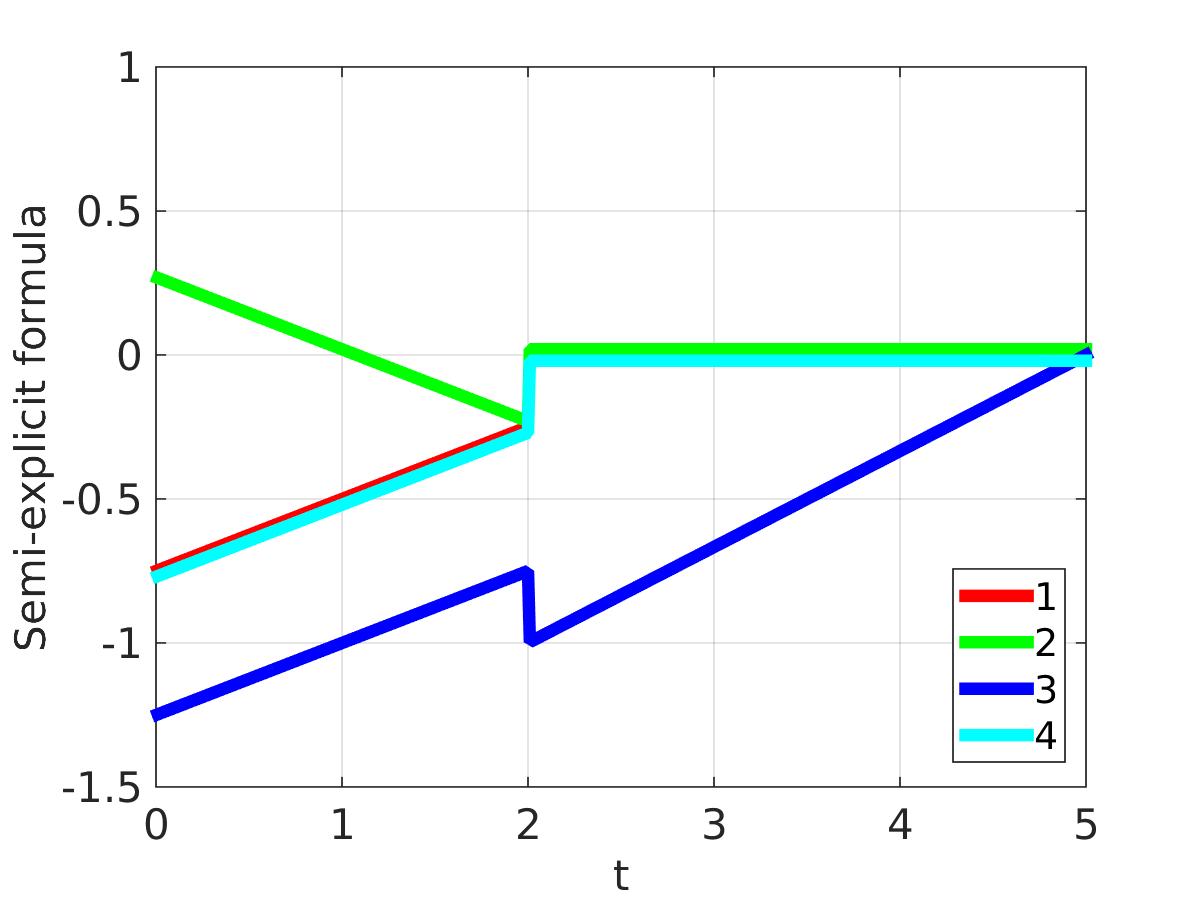} 
\subcaption{Components of \( \beta(t) \)}
\end{subfigure}
\begin{subfigure}{0.5\textwidth}
\includegraphics[width=\textwidth,keepaspectratio=true]{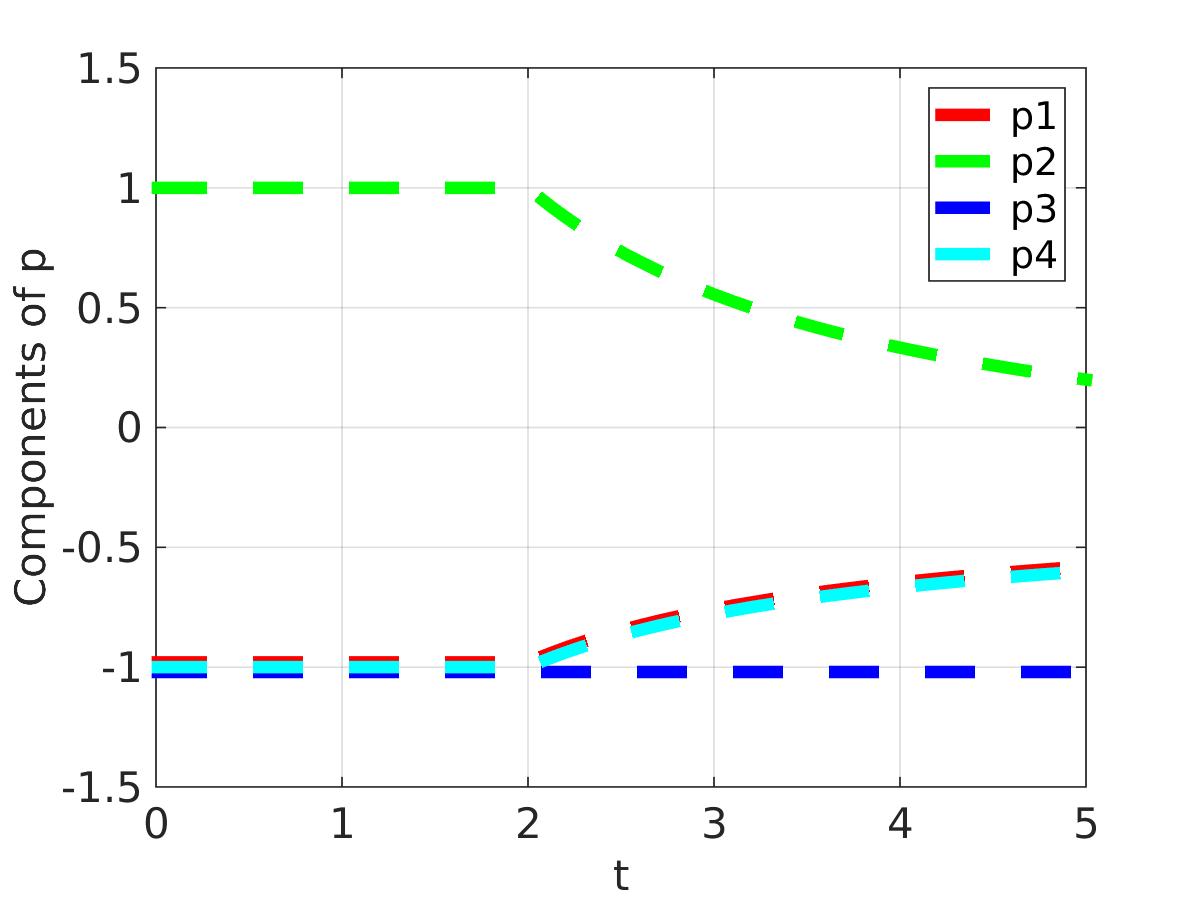} 
\subcaption{Components of \( p(t) \) }
\end{subfigure}
\caption{For \( A \) and \( f \) as in \eqref{alg_hom_other:eq_example_tib} we compare a solution path \( u(t) \) to the semi-explicit \( \beta(t) \) as in \eqref{alg_hom_other:eq_beta}. In this example, \( \beta(t) \) is not a solution path.}
\label{alg_hom_other:fig_tib}
\end{figure}

The injectivity assumption in Theorem \ref{alg_hom_other:thm_efron} is mainly needed to prevent the non-uniqueness of the solution path. Scenarios without uniqueness assumptions were studied in \cite{Tibshirani2013}, but again only under the (implicit) assumption of the one-at-a-time condition.
The results \cite[Lemma 9 and Section 3.1] {Tibshirani2013} state that a continuous and piecewise linear solution path is given by the semi-explicit formula
\begin{equation}\label{alg_hom_other:eq_beta}
\beta_{\mathcal{E}(t)}= (A_{\mathcal{E}(t)})^\dagger \left( f - ( A_{\mathcal{E}(t)}^T )^\dagger~ t~p(t)_{\mathcal{E}(t)} \right) \quad \text{and} \quad \beta_{(\mathcal{E}(t))^C} = 0 ~.
\end{equation}
{\pagebreak[3] Although  \( \beta(t) \in U_t \) for \( A \) and \( f \) as in the previous example \eqref{alg_hom_other:eq_loris_example}, this approach in general only applies under the one-at-a-time condition.}
As an example, consider 
\begin{equation}\label{alg_hom_other:eq_example_tib}
A= \begin{bmatrix}
-1 & +1 & +1 & +1 \\
+1 & -1 & +1 & +1 \\
+1 & +1 & +1 & -1 \\
\end{bmatrix},~\quad 
f=
\begin{bmatrix}
-1 \\
-3 \\
-1 
\end{bmatrix}
\quad \text{and} \quad 
t=2 ~.
\end{equation}
Then \( u(2)=\begin{bmatrix} 0 & 0 & -1 & 0 \end{bmatrix}^T \) is a solution and  \( p(2)= \begin{bmatrix} -1 & 1 & -1 & -1 \end{bmatrix}^T \) is the corresponding subgradient. But \eqref{alg_hom_other:eq_beta} yields \( \beta(2)= \begin{bmatrix} -1/4 & -1/4 & -3/4 & -1/4 \end{bmatrix} \), and thus the second component has the wrong sign. As can be seen in Figure \ref{alg_hom_other:fig_tib}, the path \( \beta(t) \) is neither continuous nor does it solve \( \beta(t) \in U_t \) for all \( t \geq 0 \). \\
 
\subsection{Adaptive Inverse Scale Space Method}
The adaptive inverse scale space (aISS) method is a fast algorithm to compute \( \ell_1\)-minimizing solutions of linear systems. Instead of calculating the minimizers of variational problems with varying regularization parameters \( t \), it computes an exact solution to the differential inclusion
\begin{equation}\label{alg_hom_other:eq_iss}
\begin{cases}
\begin{aligned}
 \partial_\nu q(t) &= A^T ( f - A v(t) ) \quad \text{with} \quad  q(t) \in \partial \| v(t) \|_1 \\
 q(0) &= 0 
\end{aligned}
\end{cases}~.
\end{equation}
The following theorem is proven in \cite[Theorem 1 and 2]{Burger2012}.
\begin{theorem}[\cite{Burger2012}]
There exists a finite sequence of times 
\begin{equation*}
0=t^0 < t^1 < t^2 < \hdots < t^K < t^{K+1} = \infty 
\end{equation*}
such that for all \( k=0,\hdots, K \)
\begin{equation*}
v(t)=v(t^k), \quad q(t)=q(t^k)+(t-t^k)~ A^T ( f- Av(t^k)), 
\end{equation*}
for \( t \in [t^k,t^{k+1}) \) is a solution of the inverse scale space flow. Here, \( v(t^k) \) is a solution of 
\begin{equation}\label{alg_hom_other:eq_iss_primal_update}
v(t^k) \in \argmin_{v \in \mathbb{R}^N} \| Av -f \|_2^2 \quad \text{s.t.} ~ q(t^k) \in \partial \| v \|_1 ~.
\end{equation}
Furthermore, \( v(t) \) is an \( \ell_1 \)-minimizing solution of \( A^TA v = A^T f \) for all \( t \geq t^K \).
\end{theorem}
The aISS method has striking similarities  to the generalized homotopy method. First, a seemingly continuous problem, i.e., a differential inclusion, can be solved completely by knowing the solution at finitely many points. While the generalized homotopy method produces a piecewise linear path \( u(t) \), the path \( v(t) \) of the aISS method is piecewise constant. \\
Second, both methods solve nonnegative least squares problems to calculate the solution path. To see this, note that in the aISS method
\begin{equation*}
q(t^k) \in \partial \| v \|_1 ~ \Leftrightarrow  ~ v_i q(t^k)_i \geq 0 ~ \forall i \text{ with } |q(t^k)_i|=1, ~ v_i = 0 ~ \forall i \text{ with } |q(t^k)_i | < 1 ~.
\end{equation*}
Although in a different context, the link between the inverse scale space flow and variational methods is also studied in  \cite{Burger2016}.

\section{Conclusions and Future Research}
In this paper, we have introduced a generalized homotopy method which computes a full solution path of \( \ell_1 \)-regularized problems in finitely many iterations. In contrast to previous homotopy methods, it provably works for an arbitrary combination of a measurement matrix and a data vector, requiring neither the uniqueness of the solution path nor the one-at-a-time condition. The backbone of the generalized homotopy method is a characterization of the set of possible directions by a nonnegative least squares problem.\\
In future research, we will extend the proposed homotopy method to arbitrary polyhedral regularizations. Furthermore, we will investigate its applicability for generalizing the ideas of nonlinear spectral decompositions considered in \cite{Gilboa2014,Burger2016} to more general data fidelity terms.

\section*{Acknowledgements}
DC and MM were supported by the ERC Starting Grant ``ConvexVision''. FK's contribution was supported by the German Science Foundation DFG in context of the Emmy Noether junior research group KR 4512/1-1 (RaSenQuaSI).


\bibliographystyle{siamplain}
\bibliography{./Misc_Arxiv}

\end{document}